\documentclass[12pt]{article}
\usepackage[top=1in,bottom=1in,left=1 in,right=.8in,footskip=0.4in]{geometry}
\usepackage{amsfonts,amsmath,amssymb,amsthm}
\usepackage[mathscr]{eucal}
\usepackage{mathtools,tabularray}
\usepackage{xfrac}
\usepackage{xcolor}
\usepackage{color,soul}
\usepackage{hyperref,url}
\usepackage{authblk}
\usepackage{orcidlink}
\usepackage{tikz}
\usepackage{tikz-cd}
\usepackage{enumerate}
\usepackage{nicematrix}
\usepackage[title]{appendix}
\usepackage{caption}
\usepackage{nicematrix}
\usepackage{calc}
\usepackage{float}
\usetikzlibrary{calc,fit}
\usepackage{layout}
\usepackage{graphicx}
\hypersetup{colorlinks=true,linkcolor=red,citecolor=blue,urlcolor=blue}
\newtheorem{theorem}{Theorem}[section]
\newtheorem{lemma}[theorem]{Lemma}
\newcommand{\W}{\mathcal{W}}
\newcommand{\T}{\mathscr{T}}
\newcommand{\Z}{\mathscr{Z}}
\newcommand{\V}{\mathcal{V}}
\newcommand{\K}{\Delta}
\newcommand{\M}{\mathcal{M}}
\theoremstyle{definition}
\newtheorem{definition}[theorem]{Definition}
\newtheorem{observation}[theorem]{Observation}
\newtheorem{remark}[theorem]{Remark}
\theoremstyle{plain}
\newtheorem{proposition}[theorem]{Proposition}

\title{Cancellation of a critical pair in discrete Morse theory and its effect on (co)boundary operators}

\author[a]{Anupam Mondal \orcidlink{0000-0002-6547-4835}}
\author[b]{Sajal Mukherjee \orcidlink{0009-0004-6959-4334}}
\author[c]{Pritam Chandra Pramanik \orcidlink{0009-0000-8332-3347}}

\affil[a,b,c]{\small Institute for Advancing Intelligence (IAI), TCG CREST,
	Kolkata -- 700091, West Bengal, India}
\affil[a,b]{\small Academy of Scientific \& Innovative Research (AcSIR),
	Ghaziabad -- 201002, Uttar Pradesh, India}
\affil[ ]{\texttt{\{anupam.mondal, sajal.mukherjee,
		pritam.pramanik.80\}@tcgcrest.org}}

\date{}
\begin{document}
	\maketitle
	\begin{abstract}
	   Discrete Morse theory helps us compute the homology groups of simplicial complexes in an efficient manner. A ``good" gradient vector field reduces the number of critical simplices, simplifying the homology calculations by reducing them to the computation of homology groups of a simpler chain complex. This homology computation hinges on an efficient enumeration of gradient trajectories. The technique of cancelling pairs of critical simplices reduces the number of critical simplices, though it also perturbs the gradient trajectories. In this article, in a purely combinatorial manner, we derive an explicit formula for computing the modified boundary operators after cancelling a critical pair, in terms of the original boundary operators. The same formula can be obtained through a sequence of elementary row operations on the original boundary operators. Thus, it eliminates the need of enumeration of the new gradient trajectories. We also obtain a similar result for coboundary operators.

		
	\end{abstract}

	\section{Introduction}
	Discrete Morse theory, introduced by Forman~\cite{forman1998}, serves as a
	combinatorial counterpart to (smooth) Morse theory (see also
	\cite{chari,forman2002,knudson,kozlov,scoville}). Over time, it has proven to be
	highly valuable across various areas of (theoretical and applied) mathematics and computer science, for example computational topology, data analysis (see \cite[Chapter 10]{TKD}), etc. At the core of this theory is the concept of a \emph{discrete Morse function} defined on a finite simplicial complex (or, a regular CW complex). This function helps us obtain a more efficient cell decomposition for the
	complex by reducing the number of simplices (or, cells) while preserving its
	homotopy-type.
	
	In practice, rather than working directly with discrete Morse functions, we
	often use an equivalent and more practical notion, namely, (discrete)
	\emph{gradient vector fields}. The homotopy-type of the complex is completely
	determined by only simplices that are \emph{critical} with respect to the assigned gradient vector field. Thus we are required to start with a gradient vector field that admits a low number of critical simplices. An \emph{optimal gradient vector field} is one that minimizes the
	number of critical simplices. However, finding such an optimal gradient vector field is an \textsf{NP-hard} problem, even for two-dimensional
	complexes~\cite{egecioglu,joswig,lewiner}.
	
	Furthermore, discrete Morse theory offers efficient methods for computing
	topological invariants such as homology groups \cite{mondal}, Betti numbers, etc.
	We remark that these computations are dependent on the choice of gradient vector field that admits a low number of critical simplices, as well as one that allows an efficient (weighted) enumeration of important combinatorial
	objects, viz., \emph{gradient trajectories}. The task of computing homology gets reduced to the computation of (discrete)
	\emph{Morse homology groups}, which are homology groups of a relatively simpler
	chain complex, known as \emph{Morse complex}.
	
	The technique of \emph{cancelling a pair of critical simplices}~\cite{forman1998,forman2002} (or, more generally \emph{simultaneous cancellations}~\cite{Hersh}) is a crucial technique that allows us to augment a given gradient vector field to a more efficient one. However, cancellation of critical pair(s) perturbs the set of critical simplices, as well as (some of) the gradient trajectories. Therefore, to compute the Morse homology groups with respect to the improved gradient vector field obtained after cancellation, one is required to
	enumerate all possible gradient trajectories with respect to the new gradient vector field. Thus, it is a natural problem of interest to 
	\begin{itemize}
		\item determine how exactly such cancellations affect gradient trajectories, and
		\item establish a formal relationship between (co)boundary operators, pre- and post-cancellations.
	\end{itemize}
	
	In this article, we address this problem and provide an explicit formula for the
	modified boundary operators (with respect to the improved gradient vector field), in
	terms of the boundary operators with respect to the gradient vector field
	before cancellation. This formula can be obtained through a sequence of elementary row operations on the matrix representations of the original boundary operators.
	\begin{theorem} \label{T1.1}
		Let $\K$ be a $d$-dimensional simplicial complex with an assigned gradient
		vector field $\V$. Let $(\sigma_0^{(k)},\tau_0^{(k-1)})$ be a cancellable
		critical pair,  for some $k \in \{1, \dots, d\}$, and $\W$ be the gradient
		vector field obtained by cancelling $(\sigma_0^{(k)},\tau_0^{(k-1)})$ from $\V$.
		Let $\partial_q^{\V}: C_q^{\V}(\K) \rightarrow C_{q-1}^{\V}(\K)$ and
		$\partial_q^{\W}: C_q^{\W}(\K) \rightarrow C_{q-1}^{\W}(\K)$ be the $q$-th
		boundary maps of the Morse complexes corresponding to $\V$ and $\W$,
		respectively. Then the following hold.
		\begin{enumerate}[(1)]
			\item For $q>k+1$ or $q<k-1$,
			\[ \partial_q^{\W}= \partial_q^{\V}.\]
			\item  If  $ \sigma_1, \dots, \sigma_n $ are the $\W$-critical $k$-simplices,
			and for any $\W$-critical $(k+1)$-simplex $\beta$, $\partial_{k+1}^{\V}(\beta)=
			\sum_{j=0}^{n} b_j\sigma_j$, then 
			\[\partial_{k+1}^{\W}(\beta)= \sum_{j=1}^{n} b_j\sigma_j.\]
			\item The operator $\partial_{k-1}^{\W}$ is the restriction of
			$\partial_{k-1}^{\V}$ to the subgroup $C_{k-1}^{\W}(\K)$, i.e.,
			\[\partial_{k-1}^{\W}=\partial_{k-1}^{\V}\big|_{C_{k-1}^{\W}(\K)}.\]
			\item Let $\sigma_0, \sigma_1, \dots, \sigma_n$ be the $\V$-critical
			$k$-simplices, and
			$\tau_0, \tau_1, \dots, \tau_m$ be the $\V$-critical $(k-1)$-simplices, and
			for all $j\in\{0, \dots , n\}$, $\partial_k^{\V}(\sigma_j)= \sum_{i=0}^m
			a_{ij}\tau_i$ . Then, for all $j\in\{1, \dots , n\}$,
			\[\partial_k^{\W}(\sigma_j)= \sum_{i=1}^m (a_{ij}-a_{00}a_{0j}a_{i0}) \cdot
			\tau_i.\]
			In other words, we may obtain (the matrix representation of)  $\partial_k^{\W}$ from $\partial_k^{\V}$ by a sequence of elementary row (or, column) operations. Let the matrix $B$ represent the boundary operator $\partial_k^{\V}$, with rows and columns indexed by $0,1,\ldots,m$ ($i$-th row corresponds to $\tau_i$) and $0,1,\ldots,n$ ($j$-th column corresponds to $\sigma_j$), respectively. We obtain a matrix 
			$B'$ from $B$ as follows
			\[B \xrightarrow[\text{for each }i\geq 1]{R_i-a_{i0}\cdot a_{00}\cdot  R_0} B'.\]
			Then $\partial_k^{\W}$ is the submatrix of $B'$, obtained by deleting the $0$-th row and the $0$-th column. 
		\end{enumerate}	
	\end{theorem}

	We also provide an analogous explicit formula (in Theorem \ref{T3.3}) for
	the modified coboundary operators. Finally in Appendix~\ref{A1}, we include an example of change in a certain boundary operator, where the cancellation technique is applied twice. 
	
	\section{Preliminaries}
	\subsection{Basics of simplicial complex and simplicial (co)homology}
	\begin{definition}[Simplicial complex]
		An (abstract) \emph{simplicial complex} $\K$ is a (finite, nonempty)
		collection of finite sets with the property that if $\sigma \in \K$ and $\tau
		\subseteq \sigma$, then $\tau \in \K$.
	\end{definition}
	We note that the empty set is always in $\K$. If $\sigma \in \K$, then
	$\sigma$ is called a \emph{simplex} or a \emph{face} of $\K$. The dimension of a simplex $\sigma$, $\operatorname{dim}(\sigma)=|\sigma|-1$. If $\operatorname{dim}(\sigma)=d$, we call $\sigma$ a $d$-\emph{dimensional simplex} (or simply, a
	$d$-\emph{simplex}). We denote a $d$-simplex $\sigma$ by $\sigma^{(d)}$ whenever
	necessary. The dimension of a simplicial complex $\K$, $\dim(\K)= \max\{\dim(\sigma):\sigma \in \K\}$. The \emph{vertex set} of a simplicial
	complex $\K$, $V(\K)$, is the collection of all elements in all the faces (i.e., $V(\K)= \cup_{\sigma \in \K}\sigma$). The
	elements of $V(\K)$ are called the \emph{vertices} of the complex $\K$. 
	
	An \emph{orientation} of a simplex is given by an ordering of its vertices,
	with two orderings defining the same orientation if and only if they differ by
	an even permutation. Let $V(\K)=\{x_0,x_1,\ldots,x_\ell \}$. We denote an
	oriented $q$-simplex $\sigma$ consisting of the vertices $x_{i_0}, x_{i_1},
	\ldots , x_{i_q}$, with $i_0 < i_1 < \ldots < i_q$, by $[x_{i_0}, x_{i_1},
	\ldots , x_{i_k}]$. In order to avoid notational complicacy, whenever necessary,
	we denote the oriented $q$-simplex $[x_{i_0}, x_{i_1}, \ldots , x_{i_q}]$ by
	$\sigma$ as well.
	
	\begin{definition}[Incidence number between simplices]
		Let $\sigma ^{(q)}=[ x_0,x_1,\ldots,x_q]$ and $\tau^{(q-1)}$ be oriented
		simplices of a simplicial complex $\K$. Then the \emph{Incidence number} between
		$\sigma$ and $\tau$ is denoted by $\langle \sigma, \tau \rangle$, and defined by
		\[\langle\sigma,\tau\rangle =\begin{cases}
			(-1)^i, & \text{ if } \tau=[x_0,\dots,\widehat{x}_i,\ldots, x_{q}], \\ 
			~~~~~0, & \text{ if } \tau \nsubseteq \sigma,	
		\end{cases}\]
		where
		$[x_0,\ldots,\widehat{x}_i,\ldots,x_q]$ is the $(q-1)$-simplex
		obtained from $\sigma$ after deleting $x_i$, with the induced orientation.
	\end{definition}
	
	For an integer $q$, let $S_q(\K)$ be the set of all $q$-dimensional simplices
	of $\K$, and $C_q(\K)$ be the free abelian group generated by $S_q(\K)$. So if
	$|S_q(\K)|=m$, then $C_q(\K)\cong\mathbb{Z}^m$. If $\dim(\K)=d$, for $q<0$ or $q>d$, then
	$C_q(\K)=\{0\}$.  We define the homomorphisms
	$\partial_q:C_{q}(\K)\rightarrow C_{q-1}(\K)$ and $\delta_q:
	C_{q-1}(\K)\rightarrow C_{q}(\K)$ by linearly extending the following maps.
	\[\partial_q(\sigma)=\sum_{\tau\in S_{q-1}(\K)}\langle\sigma,\tau \rangle \cdot
	\tau, \text{ and } \delta_q(\alpha)=\sum_{\beta\in
		S_{q}(\K)}\langle\beta,\alpha\rangle \cdot \beta,\]
	where $\sigma \in S_q$ and $\alpha \in S_{q-1}$.	
	
	So if, $\dim(\K)=d$ we have two chain complexes,
	\[\dots\rightarrow\{0\}\xrightarrow{\partial_{d+1}}
	C_d(\K)\xrightarrow{\partial_d}C_{d-1}(\K)\xrightarrow{\partial_{d-1}}\dots\xrightarrow{\partial_1}C_0(\K)\xrightarrow{\partial_0}\{0\}
	\rightarrow \dots\]
	\[\dots\rightarrow\{0\}\xrightarrow{\delta_0}
	C_0(\K)\xrightarrow{\delta_1}C_{1}(\K)\xrightarrow{\delta_{2}}\dots
	\xrightarrow{\delta_{d}}C_d(\K)\xrightarrow{\delta_{d+1}}\{0\} \rightarrow
	\dots\]
	The first one is the \emph{simplicial chain complex} and the second one is
	the \emph{simplicial cochain complex} of $\K$. For $q \in \mathbb{N}$,
	$\partial_q$ is called the \emph{$q$-th boundary map} of the simplicial chain
	complex and $\delta_q$ is called the \emph{$q$-th coboundary map} of the
	simplicial cochain complex of $\K$.
	The quotient groups
	$H_q(\K)=\sfrac{\operatorname{Ker}(\partial_q)}{\operatorname{Im}(\partial_{q+1})}$
	and
	$H^q(\K)=\sfrac{\operatorname{Ker}(\delta_{q+1})}{\operatorname{Im}(\delta_{q})}$
	are the \emph{$q$-th homology} and \emph{cohomology group} of $\K$,
	respectively.
	
	For further reading on the topics introduced in this subsection, we refer to
	\textit{Elements of Algebraic Topology}~\cite{munkres} by Munkres. 
	\subsection{Basics of discrete Morse theory}
	First, we introduce the notion of gradient vector field and some related concepts following \cite{forman1998, formancohom, forman2002}.
	\begin{definition}[Discrete vector field]
		A discrete vector field $\V$ on a simplicial complex $\K$ is a collection of
		ordered pairs of simplices of the form $(\alpha,\beta)$, such that 
		\begin{enumerate}[(i)]
			\item $\alpha \subsetneq \beta$,
			\item $\dim(\beta)=\dim(\alpha)+1$,
			\item each simplex of $\K$ is in \emph{at most} one pair of $\V$.
		\end{enumerate}
	\end{definition}
	
	If the simplex $\alpha^{(q-1)}$ is paired off with the simplex $\beta^{(q)}$ in
	$\V$ (i.e., $(\alpha, \beta) \in \V$), then we diagrammatically represent
	the pair as $\alpha \rightarrowtail \beta$ (or $\beta \leftarrowtail \alpha$).
	
	\begin{definition}[$\V$-trajectory]
		Given a discrete vector field $\V$ on a complex $\K$, a \emph{$\V$-trajectory}
		from a $q$-simplex, say $\beta_0$, to a $q$-simplex, say $\beta_r$, (or,
		alternatively, to a $(q-1)$-simplex $\alpha_{r+1}$)  is a sequence of simplices
		\[\beta_0^{(q)}, \alpha_1^{(q-1)}, \beta_1^{(q)}, \ldots, \alpha_r^{(q-1)},
		\beta_r^{(q)}
		\text{ (or, }\beta_0^{(q)}, \alpha_1^{(q-1)}, \beta_1^{(q)},\ldots,
		\alpha_r^{(q-1)}, \beta_r^{(q)}, \alpha_{r+1}^{(q-1)})\]
		such that for each $i \in \{1,\ldots,r\}$, the pair $(\alpha_i,\beta_i) \in
		\V$ and $(\alpha_i\ne)$ $\alpha_{i+1}$ $ \subsetneq \beta_i$, and $\alpha_1 \subsetneq
		\beta_0$.
	\end{definition}	
	
	We represent such a $\V$-trajectory $P$ as follows:
	\[P:\beta_0^{(q)}\rightarrow \alpha_1^{(q-1)}\rightarrowtail
	\beta_1^{(q)}\rightarrow \cdots\rightarrow \alpha_r^{(q-1)}\rightarrowtail
	\beta_r^{(q)}\]
	\[(\text{or, } P:\beta_0^{(q)}\rightarrow \alpha_1^{(q-1)}\rightarrowtail
	\beta_1^{(q)}\rightarrow \cdots\rightarrow \alpha_r^{(q-1)}\rightarrowtail
	\beta_r^{(q)} \rightarrow \alpha_{r+1}^{(q-1)})\]
	(here $``\rightarrow"$ represents set inclusion). Such a trajectory is said to be a \emph{nontrivial closed $\V$-trajectory} if
	$r > 0$ and $\beta_{r} = \beta_0$.
	
	\begin{definition} [co-$\V$-trajectory]
		A \emph{co-$\V$-trajectory} from a $q$-simplex, say $\beta_0$, to a
		$q$-simplex, say $\beta_r$, (or, alternatively to a $(q+1)$-simplex
		$\tau_{r+1}$)  is a sequence of simplices
		\[\beta_0^{(q)}, \tau_1^{(q+1)}, \beta_1^{(q)}, \ldots, \tau_r^{(q+1)},
		\beta_{r}^{(q)}
		\text{ (or, }\beta_0^{(q)}, \tau_1^{(q+1)}, \beta_1^{(q)}, \ldots,
		\tau_r^{(q+1)}, \beta_{r}^{(q)}, \tau_{r+1}^{(q+1)})\]
		such that for each $i \in \{1,\ldots,r\}$, the pair $(\beta_i,\tau_i) \in \V$
		and $(\beta_i \ne)$ $\beta_{i-1} \subsetneq \tau_{i}$, and $\beta_r \subsetneq
		\tau_{r+1}$.
	\end{definition}
	
	We represent such a co-$\V$-trajectory $Q$ as follows:
	\[Q:\beta_0^{(q)} \leftarrow \tau_1^{(q+1)}\leftarrowtail
	\beta_1^{(q)}\leftarrow \cdots\leftarrow \tau_r^{(q+1)}\leftarrowtail
	\beta_{r}^{(q)}\]
	\[(\text{or, }Q:\beta_0^{(q)}\leftarrow \tau_1^{(q+1)}\leftarrowtail
	\beta_1^{(q)}\leftarrow \cdots\leftarrow \tau_r^{(q+1)}\leftarrowtail
	\beta_{r}^{(q)} \leftarrow \tau_{r+1}^{(q+1)}).\]
	Such a trajectory is said to be a \emph{nontrivial closed 
		co-$\V$-trajectory} if $r > 0$ and $\beta_{r} = \beta_0$.
	
     When
	$P:\beta_0^{(q)}\rightarrow \alpha_1^{(q-1)}\rightarrowtail
	\beta_1^{(q)}\rightarrow \cdots\rightarrow \alpha_r^{(q-1)}\rightarrowtail
	\beta_r^{(q)}$ is a $\V$-trajectory, the \emph{weight} of $P$ is defined by
	\[w_{\V}(P)=\prod_{i=1}^{r}\left(-\langle \beta_{i-1},\alpha_i\rangle \langle
	\beta_i, \alpha_i \rangle \right).\]
	Similarly, when $P:\beta_0^{(q)}\rightarrow \alpha_1^{(q-1)}\rightarrowtail
	\beta_1^{(q)}\rightarrow \cdots\rightarrow \alpha_r^{(q-1)}\rightarrowtail
	\beta_r^{(q)} \rightarrow \alpha_{r+1}^{(q-1)}$,
	\[ w_{\V}(P)=\left(\prod_{i=1}^{r}\left(-\langle \beta_{i-1},\alpha_i\rangle
	\langle \beta_i, \alpha_i \rangle \right) \right)\cdot \langle \beta_r,
	\alpha_{r+1}\rangle .\]
 Abusing the notation, we drop the parameter $\V$ from $w_{\V}(P)$ 	for the sake of simplicity; it should be understood from the context.
	
	When $Q:\beta_0^{(q)} \leftarrow \tau_1^{(q+1)}\leftarrowtail
	\beta_1^{(q)}\leftarrow \cdots\leftarrow \tau_r^{(q+1)}\leftarrowtail
	\beta_{r}^{(q)}$ is a co-$\V$-trajectory, the \emph{weight} of $Q$ is defined by
	\[w(Q)=\prod_{i=1}^{r}(-\langle \tau_i, \beta_{i-1}\rangle \langle \tau_i
	,\beta_i \rangle)\]
	Similarly, when $Q:\beta_0^{(q)}\leftarrow \tau_1^{(q+1)}\leftarrowtail
	\beta_1^{(q)}\leftarrow \cdots\leftarrow \tau_r^{(q+1)}\leftarrowtail
	\beta_{r}^{(q)} \leftarrow \tau_{r+1}^{(q+1)}$,
	\[ w(Q)= \left( \prod_{i=1}^{r}(-\langle \tau_i, \beta_{i-1}\rangle \langle
	\tau_i ,\beta_i \rangle) \right) \cdot \langle \tau_{r+1}, \beta_{r}\rangle.\]
	The (co-)trajectory $P$ is said to be \emph{trivial}
	if $r=0$. In this case, $w(P)$ is defined to be $1$. We note that, the weight of any (co-)trajectory is either 1 or $-1$.
	
	We observe that, for a given discrete vector field $\V$ on a simplicial complex
	$\K$, there exists no nontrivial closed $\V$-trajectory if and only if there
	exists no nontrivial closed co-$\V$-trajectory.
	\begin{definition}[Gradient vector field]
		A \emph{gradient vector field} on a complex $\K$ is a discrete vector field
		$\V$ on $\K$ which does not admit nontrivial closed $\V$-trajectories
		(equivalently, nontrivial closed co-$\V$-trajectories).
	\end{definition}
	We remark that a gradient vector field is also referred to as an \emph{acyclic
		matching} in literature~\cite{chari}. We may also call a $\V$-trajectory a
	\emph{gradient trajectory}, when $\V$ is a gradient vector field.
	
	Let $P: \eta_0,\eta_1,\ldots,\eta_r$ be a gradient trajectory (or, a sequence of
	simplices in general). For $(0\leq) $ $i <j $ $(\leq r)$, we denote the (sub)sequence of
	simplices from $\eta_i$ to $\eta_j$ by $\eta_iP\eta_j$, i.e., 
	\[\eta_iP\eta_j: \eta_i,\eta_{(i+1)},\ldots,\eta_j.\]  
	Now, suppose
	\[P_1: \eta_0, \eta_1,\ldots,\eta_i,\ldots,\eta_r, \text{ and }
	P_2: \eta_0^\prime, \eta_1^\prime,\ldots,\eta_j^\prime,\ldots,\eta_s^\prime\]
	are two gradient trajectories, such that $\eta_i=\eta_j^\prime$. For $\ell\leq
	i$ and $t\geq j$, we denote the concatenation of $\eta_{\ell}P_1\eta_i$ and
	$\eta_{j}^\prime P_2\eta_t^\prime$  by $\eta_{\ell}P_1\eta_iP_2\eta_t^\prime$
	(or, $\eta_{\ell}P_1\eta^\prime_jP_2\eta^\prime_t$), i.e., 
	\[\eta_{\ell}P_1\eta_iP_2\eta_t^\prime:
	\eta_{\ell},\eta_{\ell+1},\ldots,\eta_i(=\eta_j^\prime),\eta_{j+1}^\prime,\ldots,\eta^\prime_t.\]
	
	We use similar intuitive notation, e.g., $\eta_{\ell}P_1\eta_jP_2\eta_t^\prime
	P_3 \eta^{\prime\prime}_u$ (defined associatively), for three or more
	trajectories or sequences.
	\begin{definition}[Critical simplex]
		Let $\V$ be a gradient vector field on a complex $\K$. A nonempty simplex
		$\sigma$ is said to be a \emph{$\V$-critical simplex} (or, a \emph{critical
			simplex}, when the gradient vector field is clear from the context) if one of the following
		holds:
		\begin{enumerate}[(i)]
			\item $\sigma$ does not appear in any pair of $\V$, or
			\item $\sigma$ is a $0$-simplex and $(\emptyset,\sigma) \in \V$.
		\end{enumerate}
	 Also, $\operatorname{Crit}_q^{\V}(\K)$ is the set of all $q$-dimensional $\V$-critical simplices.
	\end{definition}
	The following important theorem presents a method for reducing  the number of critical
	simplices in a given gradient vector field on a simplicial complex.
	\begin{theorem}[Cancellation of a pair of critical simplices,
		\cite{forman1998, forman2002}] \label{T2.7}
		Suppose $\V$ is a gradient vector field
		on a $d$-dimensional simplicial complex $\K$, such that $\sigma^{(k)}$ and
		$\tau^{(k-1)}$ are critical, for some $k \in \{1,\cdots,d\}$. If there is a
		unique $\V$-trajectory from $\sigma^{(k)}$ to $\tau^{(k-1)}$ (we call such a
		pair $(\sigma^{(k)},\tau^{(k-1)})$ a `cancellable critical pair'), then there is
		a
		gradient vector field $\W$ on $\K$ such that the critical
		simplices with respect to $\W$ remain the same, except that
		$\sigma$ and $\tau$ are no longer critical. Moreover, $\W$ is
		same as $\V$ except along the unique $\V$-trajectory from $\sigma$ to 
		$\tau$.
	\end{theorem}
	A sketch of a proof is as follows. Let the unique $\V$-trajectory $P$ from to
	$\sigma$ to $\tau$ be
	\[P:(\sigma^{(k)}=)\beta_0^{(k)},\alpha_1^{(k-1)}, \beta_1^{(k)}, \ldots,
	\alpha_r^{(k-1)}, \beta_r^{(k)}, \alpha_{r+1}^{(k-1)}(=\tau^{(k-1)}).\]

	We get $\W$ from $\V$ by removing $(\alpha_i, \beta_i)$, for $i \in \{1, \dots , r\}$, and adding $(\alpha_{i+1},\beta_i)$, for $ i \in \{0, \dots , r\}$, i.e.,
	\[\W = (\V \setminus \{(\alpha_i, \beta_i) : i \in \{1, \dots , r\}\})
	\sqcup\{(\alpha_{i+1},\beta_i): i \in \{0, \dots , r\}\}\]
	($\sqcup$ denotes the union of disjoint sets).
	
		In other words we have the following diagram:
	\[P:(\sigma^{(k)}=)\beta_0^{(k)}\rightarrow \alpha_1^{(k-1)}\rightarrowtail
	\beta_1^{(k)}\rightarrow \cdots\rightarrow \alpha_r^{(k-1)}\rightarrowtail
	\beta_r^{(k)} \rightarrow \alpha_{r+1}^{(k-1)}(=\tau^{(k-1)}).\]
	
	We obtain $\W$ from $\V$ by flipping the arrows, with `$\rightarrow$' becoming
	`$\leftarrowtail$' and `$\rightarrowtail$' becoming `$\leftarrow$', as follows.
	\[(\sigma^{(k)}=)\beta_0^{(k)}\leftarrowtail \alpha_1^{(k-1)}\leftarrow
	\beta_1^{(k)}\leftarrowtail \cdots\leftarrowtail \alpha_r^{(k-1)}\leftarrow
	\beta_r^{(k)} \leftarrowtail \alpha_{r+1}^{(k-1)}(=\tau^{(k-1)}).\]

	The uniqueness of the
	$\V$-trajectory from $\sigma$ to $\tau$ guarantees that $\W$ is also a gradient
	vector field on $\K$. This also implies that $\sigma$ and $\tau$ are not
	critical with respect to $\W$,
	while the criticality of all other simplices remains unchanged.
	
	The following allows us to apply the technique above to cancel several
	pairs of critical simplices simultaneously.
	
	\begin{theorem} \cite{Hersh} \label{t2.9}
		Let $\V$ be a gradient vector field on a complex $\K$ such that, for $i \in
		\{1,\dots, t\}$,
		there is a unique $\V$-trajectory $P_i$ from a critical $k_i$-simplex
		$\sigma_i$ to the critical
		$(k_{i}-1)$-simplex $\tau_i$. If there is no non-identity permutation $\pi$
		of $t$ elements such that there is a $\V$-trajectory from $\sigma_i$ to
		$\tau_{\pi(i)}$, for all $i \in \{1,\cdots , t\}$, then reversing all the
		$\V$-trajectories $P_i$
		(to cancel the critical pair $\tau_i$ and $\sigma_i$) would still produce a
		gradient vector field on $\K$.	
	\end{theorem}
    
     Here we remark that Benedetti, Lutz, et al.\ proposed a methodology for searching optimal discrete gradient vector
     fields with a random heuristic~\cite{adi,benedetti1,benedetti2}. This approach turned out to be successful, even in some cases with a large input
     size.	
	\subsection{Morse complex and co-Morse complex}
	
	For a given gradient vector field $\V$ on a simplicial complex $\K$, this
	section introduces two chain complexes: the (discrete) Morse complex and the
	co-Morse complex. The Morse complex is homotopy equivalent to the simplicial
	chain complex, while the co-Morse complex is homotopy equivalent to the
	simplicial co-chain complex of $\K$. Consequently, their homology and cohomology
	groups are identical. Both of these chain complexes are generated by the
	$\V$-critical simplices of $\K$. It
	simplifies the computation of homology and cohomology compared to the original
	chain and co-chain complexes.
	We remark that the Morse complex is also referred to as \emph{combinatorial
		Thom--Smale Complex} in literature \cite{gallais}.
	
	Let $C_q^{\V}{(\K)}$ be the free abelian group
	generated by $\operatorname{Crit}_q^{\V}(\K)$. We define homomorphisms,
	$\partial_q^{\V}:C_q^{\V}(\K) \rightarrow C_{q-1}^{\V}(\K)$ and
	$\delta_q^{\V}:C_{q-1}^{\V}(\K) \rightarrow C_{q}^{\V}(\K)$ by linearly
	extending the following maps.

	\[\partial^{\V}_q(\beta)= \sum_{\alpha \in \operatorname{Crit^{\V}_{q-1}}(\K)}
	\left( \sum_{\substack {P:P \text{ is a }\\ \V \text{-traj. from } \\ \beta
			\text{ to } \alpha }} w(P) \right) \cdot \alpha, ~~~~\delta^{\V}_q(\sigma)=
	\sum_{\tau \in \operatorname{Crit^{\V}_{q}}(\K)} \left( \sum_{\substack {Q: Q
			\text{ is a }\\ \text{co-}\V\text{-traj. from } \\ \sigma \text{ to } \tau }}
	w(Q) \right) \cdot \tau,\] 
	where $\beta \in \operatorname{Crit^{\V}_{q}}(\K)$ and $\sigma\in
	\operatorname{Crit^{\V}_{q-1}}(\K)$.
	
	\begin{lemma}\cite{formancohom, gallais} \label{2.4}for all $q\in
		\mathbb{N}$,
		$\mathrm{(i)}$ $\partial_{q-1}^{\V}\circ \partial_{q}^{\V}=0$,
		$\mathrm{(ii)}$	$\delta_{q}^{\V}\circ \delta_{q-1}^{\V}=0$.
		
	\end{lemma}
	Lemma \ref{2.4} lets us define the \emph{Morse Complex},
	\[\dots\rightarrow\{0\}\xrightarrow{\partial^{\V}_{d+1}}
	C_d^{\V}(\K)\xrightarrow{\partial^{\V}_d}C^{\V}_{d-1}(\K)\xrightarrow{\partial^{\V}_{d-1}}\dots\rightarrow
	C^{\V}_1(\K)\xrightarrow{\partial^{\V}_1}C^{\V}_0(\K)\xrightarrow{\partial^{\V}_0}\{0\}
	\rightarrow \cdots,\]
	and the \emph{co-Morse Complex},  
	\[\dots\rightarrow\{0\}\xrightarrow{\delta^{\V}_0}
	C^{\V}_0(\K)\xrightarrow{\delta^{\V}_1}C^{\V}_{1}(\K)\xrightarrow{\delta^{\V}_{2}}\dots\rightarrow
	C^{\V}_{d-1}(\K)\xrightarrow{\delta^{\V}_{d}}C^{\V}_d(\K)\xrightarrow{\delta^{\V}_{d+1}}\{0\}
	\rightarrow \cdots.\]
	The map $\partial_q^{\V}$ is called the $q$-th boundary map of the Morse
	complex and $\delta_q^{\V}$ is called the $q$-th coboundary map of the co-Morse
	complex. 
	By Lemma \ref{2.4}, $\operatorname{Im}(\partial^{\V}_{q+1})$ is a subgroup
	of $\operatorname{Ker}(\partial^{\V}_q)$ and
	$\operatorname{Im}(\delta^{\V}_{q})$ is a subgroup of
	$\operatorname{Ker}(\delta^{\V}_{q+1})$. The quotient groups
	$H^{\V}_q(\K)=\sfrac{\operatorname{Ker}(\partial^{\V}_q)}{\operatorname{Im}(\partial^{\V}_{q+1})}$
	and
	$H^{\V,q}(\K)=\sfrac{\operatorname{Ker}(\delta^{\V}_{q+1})}{\operatorname{Im}(\delta^{\V}_{q})}$
	are called the \emph{$q$-th Morse homology group} and \emph{$q$-th Morse
		cohomology group} of $\K$, respectively.
	
	\begin{theorem}\cite{forman1998, formancohom,gallais}
		Let $\K$ be a simplicial complex, and $\V$ be a gradient vector field
		defined on it. Then for any $q\in \mathbb{N}$, $\mathrm{(i)}$ $H_q(\K)$ and
		$H_q^{\V}(\K)$ are isomorphic, and $\mathrm{(ii)}$ $H^q(\K)$ and $H^{\V,q}(\K)$
		are isomorphic.
	\end{theorem}
	
	\section{Proof of the main theorem}
	
	Let $\K$ be a $d$-dimensional simplicial complex, and $\V$ be a gradient vector
	field defined on $\K$ such that, for some $k \in \{1, \dots, d\}$,
	$(\sigma_0^{(k)},\tau_0^{(k-1)})$ is a cancellable critical pair. Let $\W$ be
	the gradient vector field obtained by cancelling
	$(\sigma_0^{(k)},\tau_0^{(k-1)})$ from $\V$. Let $\partial_q^{\V}: C_q^{\V}
	\rightarrow C_{q-1}^{\V}$ and $\partial_q^{\W}: C_q^{\W} \rightarrow
	C_{q-1}^{\W}$ be the $q$-th boundary maps of the Morse complex of $\K$
	corresponding to the gradient vector fields $\V$ and $\W$, respectively. Note
	that, when $q \notin \{k,k-1\}$,
	$\operatorname{Crit_q^{\W}}(\K)=\operatorname{Crit_q^{\V}}(\K)$, and thus
	$C_q^{\W}(\K)=C_q^{\V}(\K)$.
	\begin{center}
		\begin{tikzpicture}
			\node at (0,0) {\small{$C_{k+2}^{\V}(\K)$}};
			\node at (2.5,0) {\small{$C_{k+1}^{\V}(\K)$}};
			\node at (5,0) {\small{$C_{k}^{\V}(\K)$}};
			\node at (7.5,0) {\small{$C_{k-1}^{\V}(\K)$}};
			\node at (10,0) {\small{$C_{k-2}^{\V}(\K)$}};
			\node at (-2.4,0) {$\cdot \hspace{0.1cm} \cdot \hspace{0.1cm} \cdot$};
			\node at (12.4,0) {$\cdot \hspace{0.1cm} \cdot \hspace{0.1cm} \cdot$};
			
			\draw[->]  (-1.8,0)--(-.8,0);
			\draw[->]  (0.8,0)--(1.7,0);
			\draw[->]  (3.3,0)--(4.3,0);
			\draw[->]  (5.7,0)--(6.7,0);
			\draw[->]  (8.3,0)--(9.2,0);
			\draw[->]  (10.8,0)--(11.9,0);
			
			\node at (1.25,0.3) {\scriptsize{$\partial_{k+2}^{\V}$}};
			\node at (3.8,0.3) {\scriptsize{$\partial_{k+1}^{\V}$}};
			\node at (6.2,0.3) {\scriptsize{$\partial_{k}^{\V}$}};
			\node at (8.75,0.3) {\scriptsize{$\partial_{k-1}^{\V}$}};
			\node at (11.35,0.3) {\scriptsize{$\partial_{k-2}^{\V}$}};
			
			\node at (5.2,.8) {\small{$\sigma_0$}};
			\node at (7.5,.8) {\small{$\tau_0$}};
			
			\node at (5.2,.5) {\rotatebox[origin=c]{270}{$\in$}};
			\node at (7.5,0.5) {\rotatebox[origin=c]{270}{$\in$}};
		\end{tikzpicture}
	\end{center}
	
	The image of any $\W$-critical $q$-simplex, under the map
	$\partial_q^{\W}$, depends only on the $\W$-trajectories from that simplex to the
	$\W$-critical $(q-1)$-simplices. Theorem \ref{T2.7} implies that, when $q \neq
	k$, the $\W$-trajectories and the $\V$-trajectories from a $\W$-critical
	$q$-simplex to a $\W$-critical $(q-1)$-simplex are identical. This implies, when
	$q\neq k$, for any $\W$-critical $q$-simplex $\sigma$ and for any $\W$-critical
	$(q-1)$-simplex $\tau$, the coefficient of $\tau$ in $\partial_q^{\W}(\sigma)$
	is the same as the coefficient of $\tau$ in $\partial_q^{\V}(\sigma)$.  So, the
	following can be deduced.
	\begin{enumerate}
		\item For $q>k+1$ or $q<k-1$,
		\[ \partial_q^{\W}= \partial_q^{\V}.\]
		
		\item  If  $\operatorname{Crit}^{\W}_k(\K)= \{ \sigma_1, \dots, \sigma_n \}$
		and for any $\beta \in \operatorname{Crit}_{k+1}^{\W}(\K)$,
		$\partial_{k+1}^{\V}(\beta)= \sum_{j=0}^{n} b_j\sigma_j$, then, 
		\[\partial_{k+1}^{\W}(\beta)= \sum_{j=1}^{n} b_j\sigma_j.\]
		
		\item The boundary operator $\partial_{k-1}^{\W}$ is the restriction of
		$\partial_{k-1}^{\V}$ to the subgroup $C_{k-1}^{\W}(\K)$, i.e.,
		\[\partial_{k-1}^{\W}=\partial_{k-1}^{\V}\big|_{C_{k-1}^{\W}(\K)}\]
		(note that, $\operatorname{Crit}_{k-1}^{\W}(\K)=
		\operatorname{Crit}_{k-1}^{\V}(\K) \setminus \{\tau_0\}$).
	\end{enumerate}

	The following proposition covers the case $q=k$, i.e., it establishes a
	relation between the two $k$-th boundary maps $\partial_{k}^{\W}$ and
	$\partial_{k}^{\V}$.
	
	\begin{proposition} \label{P3.1.}
		Let $\K$ be a $d$-dimensional simplicial complex with an assigned gradient
		vector field $\V$. Let $(\sigma_0^{(k)},\tau_0^{(k-1)})$ be a cancellable
		critical pair,  for some $k \in \{1, \dots, d\}$, and $\W$ be the gradient
		vector field obtained by cancelling $(\sigma_0^{(k)},\tau_0^{(k-1)})$ from $\V$.
		Let $\partial_q^{\V}: C_q^{\V}(\K) \rightarrow C_{q-1}^{\V}(\K)$ and
		$\partial_q^{\W}: C_q^{\W}(\K) \rightarrow C_{q-1}^{\W}(\K)$ be the $q$-th
		boundary maps of the Morse complexes of $\K$ corresponding to $\V$ and $\W$,
		respectively. 
		Let $\operatorname{Crit}^{\V}_{k}(\K)= \{ \sigma_0, \sigma_1, \dots, \sigma_n
		\}$, 
		$\operatorname{Crit}^{\V}_{k-1}(\K)= \{ \tau_0, \tau_1, \dots, \tau_m \}$, and
		for all $j\in\{0, \dots , n\}$, $\partial_k^{\V}(\sigma_j)= \sum_{i=0}^m
		a_{ij}\tau_i$. Then, for all $j\in\{1, \dots , n\}$,
		\[\partial_k^{\W}(\sigma_j)= \sum_{i=1}^m (a_{ij}-a_{00}a_{0j}a_{i0}) \cdot
		\tau_i.\]
		
	\end{proposition}
	
	Since, $(\sigma_0, \tau_0)$ is a cancellable critical pair, there is a
	unique $\V$-trajectory $P_0$ from $\sigma_0$ to $\tau_0$ as follows.
	
	\[P_0: (\sigma_0=) \beta^{(k)}_0, \alpha^{(k-1)}_1, \beta^{(k)}_1, \cdots,
	\alpha^{(k-1)}_r, \beta^{(k)}_r, \alpha^{(k-1)}_{r+1}(=\tau_0).\]
	
	Let $\overline{P}_0$ be the sequence obtained by reversing the order of
	simplices in $P_0$, i.e.,
	
	\[\overline{P}_0: (\tau_0=) \alpha^{(k-1)}_{r+1} , \beta^{(k)}_r ,
	\alpha^{(k-1)}_r , \cdots, \beta^{(k)}_1 , \alpha^{(k-1)}_1 ,
	\beta^{(k)}_{0}(=\sigma_0).\]
	
	In the following diagram we represent $P_0$ with black arrows and
	$\overline{P}_0$ by red arrows (right to left).

	\begin{tikzpicture}
		\node at (0,0) {$P_0:(\sigma_0=)\beta_0^{(k)}$};
		\node at (3,0) {$\alpha_{1}^{(k-1)}$};
		\node at (5.2,0) {$\beta_{1}^{(k)}$};
		\node at (7.1,0) {$\cdots$};
		\node at (9.1,0) {$\alpha_{r}^{(k-1)}$};
		\node at (11,0) {$\beta_{r}^{(k)}$};
		\node at (13.5,0) {$\alpha_{r+1}^{(k-1)}(=\tau_0)$};

		\node at (7.2,-.7) {\textcolor{red}{$\overline{P}_0$}};
		\node at (7.2,.7) {$P_0$};
		
		\draw[->] (1.5,0.1)--(2.3,0.1);
		\draw[>->] (3.7,0.1)--(4.5,0.1);
		\draw[->] (5.8,0.1)--(6.6,0.1);
		\draw[->] (7.5,0.1)--(8.3,0.1);
		\draw[>->] (9.8,0.1)--(10.6,0.1);
		\draw[->] (11.5,0.1)--(12.3,0.1);
		
		\draw[<-<,red] (1.5,-0.1)--(2.3,-0.1);
		\draw[<-,red] (3.7,-0.1)--(4.5,-0.1);
		\draw[<-<,red] (5.8,-0.1)--(6.6,-0.1);
		\draw[<-<,red] (7.5,-0.1)--(8.3,-0.1);
		\draw[<-,red] (9.8,-0.1)--(10.6,-0.1);
		\draw[<-<,red] (11.5,-0.1)--(12.3,-0.1);

	\end{tikzpicture}
	
	Before going to the proof of Proposition
	\ref{P3.1.}, we discuss the  relationship between the $\W$-trajectories and the
	$\V$-trajectories in $\K$.

	\begin{observation} \label{ob1} For $\sigma_j\in \operatorname{Crit}^{\W}_k(\K)$ and $\tau_i\in \operatorname{Crit}^{\W}_{k-1}(\K)$, we have the following.
		\begin{enumerate}[(1)]
			\item  Any $\V$-trajectory which does not involve any simplices from $P_0$, is also a $\W$-trajectory, and vice versa.
			
				\item Any $\V$-trajectory from $\sigma_0$ to $\tau_i$, after starting from $\sigma_0$, follows along $P_0$ to some $k$-simplex $\beta_s$, where $s\in \{0,\ldots,r\}$ (see Figure~\ref{f1}), and then leaves $P_0$ to reach $\tau_i$. It is not possible for such a trajectory to return to $P_0$ after $\beta_s$, otherwise it would violate the uniqueness of $P_0$. 
			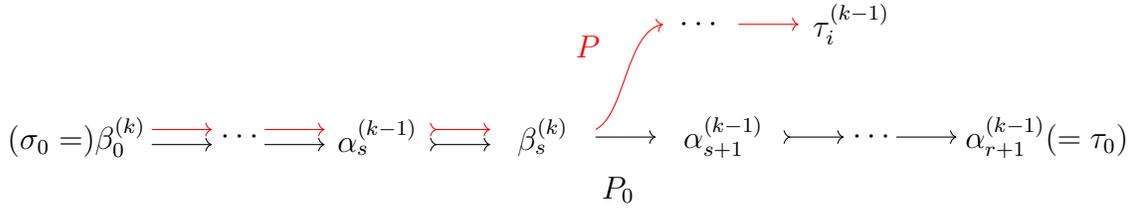
\begin{figure}[H]
				\hfill
				\begin{center}
					\begin{tikzpicture}
						\node at (0.5,0) {$(\sigma_0=)\beta_0^{(k)}$};
						\node at (2.7,0) {$\cdots$};
						\node at (4.5,0) {$\alpha_{s}^{(k-1)}$};
						\node at (6.7,0) {$\beta_{s}^{(k)}$};
						\node at (9.1,0) {$\alpha_{s+1}^{(k-1)}$};
						\node at (11.1,0) {$\cdots$};
						\node at (13.4,0) {$\alpha_{r+1}^{(k-1)}(=\tau_0)$};
						
						\node at (8.8,1.5) {$\cdots$}; 
						\node at (10.8,1.5) {$\tau_i^{(k-1)}$};
						
						\node at (7.3,1.2) {\textcolor{red}{$P$}};
						\node at (7.7,-.7) {$P_0$};
						
						\draw[->] (1.5,-0.1)--(2.3,-0.1);
						\draw[->] (3,-0.1)--(3.8,-0.1);
						\draw[>->] (5.2,-0.1)--(6,-0.1);
						\draw[->] (7.4,0)--(8.2,0);
						\draw[>->] (9.9,0)--(10.7,0);
						\draw[->] (11.4,0)--(12.2,0);
						
						\draw[->,red] (7.4,0.1)..controls (7.8,0.3) and (7.8,1.4) .. (8.3,1.5);
						\draw[->,red] (9.3,1.5)--(10.1,1.5);
						
						\draw[->,red] (1.5,0.1)--(2.3,0.1);
						\draw[->,red] (3,0.1)--(3.8,0.1);
						\draw[>->,red] (5.2,0.1)--(6,0.1);
						
					\end{tikzpicture}
				\end{center}
				\caption{ $P$ is a $\V$-trajectory (red) that starts from $\sigma_0$, 
					follows along $P_0$ (black) to $\beta_s$ and then leaves $P_0$ to reach
					$\tau_i$.  }
				\label{f1}
			\end{figure}
		
			\item Any $\W$-trajectory from $\sigma_j$ to $\sigma_0$ first meets $P_0$ at some $(k-1)$-simplex $\alpha_t$, where $t\in \{1,\ldots,r+1\}$ (see Figure \ref{f2}), and follows
			along $\overline{P}_0$ to $\sigma_0$. It is not possible for such a trajectory
			to leave $\overline{P}_0$ after $\alpha_t$, otherwise it would give rise to a
			nontrivial closed $\V$-trajectory, violating the \emph{acyclicity} of the
			gradient vector field $\V$.	
			\begin{figure}[H] 
				\begin{center}
					\begin{tikzpicture}
						\node at (0.5,0) {$(\sigma_0=)\beta_0^{(k)}$};
						\node at (2.7,0) {$\cdots$};
						\node at (4.7,0) {$\beta_{t-1}^{(k)}$};
						\node at (7.1,0) {$\alpha_{t}^{(k-1)}$};
						\node at (9.4,0) {$\beta_{t}^{(k)}$};
						\node at (11.2,0) {$\cdots$};
						\node at (13.5,0) {$\alpha_{r+1}^{(k-1)}(=\tau_0)$};
						\node at (4.3,-1.5) {$\sigma_j^{(k)}$};
						\node at (5.9,-1.5) {$\cdots$};
						
						\node at (7.2,-1.2) {\textcolor{red}{$P$}};
						\node at (8,.7) {$P_0$};
						
						\draw[->] (1.5,0.1)--(2.3,0.1);
						\draw[>->] (3,0.1)--(3.8,0.1);
						\draw[->] (5.5,0.1)--(6.3,0.1);
						\draw[>->] (7.8,0)--(8.6,0);
						\draw[->] (10,0)--(10.8,0);
						\draw[->] (11.5,0)--(12.3,0);
						\draw[->,red] (4.7,-1.5)--(5.5,-1.5);
						\draw[->,red] (6.3,-1.5) .. controls (6.86,-1.47) and (7.12,-1) ..
						(7,-0.5);
						\draw[<-<,red] (1.5,-0.1)--(2.3,-0.1);
						\draw[<-,red] (3,-0.1)--(3.8,-0.1);
						\draw[<-<,red] (5.5,-0.1)--(6.3,-0.1);
					\end{tikzpicture}
				\end{center}
				\caption{ $P$ is a $\W$-trajectory (red) from $\sigma_j$ to $\sigma_0$,
					which meets $P_0$ (black) at $\alpha_t$, then follows along $\overline{P}_0$ to
					$\sigma_0$.  }
				\label{f2}
			\end{figure}
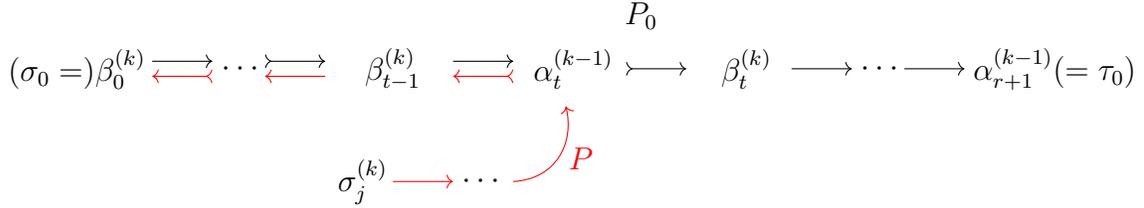

			\item Any $\V$-trajectory from  $\sigma_j$ to $\tau_i$, which involves some
			simplices from $P_0$, has the following property. After starting from $\sigma_j$,
			it meets $P_0$ at some $(k-1)$-simplex $\alpha_t$, where $t\in \{1,\ldots,r\}$, then follows along $P_0$
			to some $k$-simplex $\beta_{s}$, with $ t \leq s \leq r$, and then leaves $P_0$ and
			ends at $\tau_i$ (see Figure \ref{f3}).  Thus, $P$
			contains only the simplices $\beta_{t}$,$\ldots$, $\beta_{s}$ and
			$\alpha_{t}$,$\ldots$, $\alpha_{s}$ from $P_0$.
			
			\begin{figure}[H]
				\hfill
				\begin{center}
					\begin{tikzpicture}
						\node at (0.5,0) {$(\sigma_0=)\beta_0^{(k)}$};
						\node at (2.7,0) {$\cdots$};
						\node at (4.55,0) {$\alpha_t^{(k-1)}$};
						
						\node at (6.5,0) {$\cdots$};
						\node at (8.2,0) {$\beta_s^{(k)}$};
						\node at (10,0) {$\cdots$};
						\node at (12.4,0) {$\alpha_{r+1}^{(k-1)}(=\tau_0)$};

						\node at (10.1,1.5) {$\cdots$}; 
						\node at (12,1.5) {$\tau_i^{(k-1)}$};

						\node at (1.8,-1.5) {$\sigma_j^{(k)}$};
						\node at (3.4,-1.5) {$\cdots$};
						\node at (3,.5) {$P_0$};
						\node at (6.5,.7) {\textcolor{red}{$P$}};
						
						\draw[->] (1.5,0)--(2.3,0);
						\draw[->] (3,0)--(3.8,0);
						\draw[>->] (5.2,-0.05)--(6,-0.05);
						\draw[>->] (6.9,-0.05)--(7.7,-0.05);
						\draw[->] (8.7,0)--(9.5,0);
						\draw[->] (10.4,0)--(11.2,0);
						
						\draw[->,red] (2.2,-1.5)--(3,-1.5);
						\draw[->,red] (3.8,-1.5) .. controls (4.36,-1.47).. (4.4,-0.5);

						\draw[>->,red] (5.2,0.15)--(6,0.15);
						\draw[>->, red] (6.9,0.15)--(7.7,0.15);

						\draw[->,red] (8.7,0.1)..controls (9,0.3) and (9,1.4) .. (9.6,1.5);
						\draw[->,red] (10.5,1.5)--(11.3,1.5);

					\end{tikzpicture}
				\end{center}
				\caption{ $P$ (red) is a $\V$-trajectory from  $\sigma_j$ to $\tau_i$ that
					meets $P_0$ (black) at $\alpha_t$, follows along $P_0$ to $\beta_{s}$ and then
					leaves the trajectory $P_0$ to reach $\tau_i$.  }
				\label{f3}
			\end{figure}
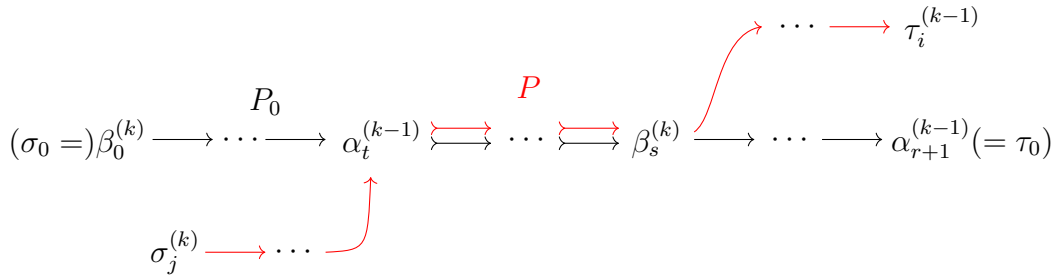

			\item  Any $\W$-trajectory from  $\sigma_j$ to $\tau_i$, which involves
			some simplices from $P_0$, has the following property. After starting from
			$\sigma_j$, it meets $P_0$ at some $(k-1)$-simplex $\alpha_t$, where  $t \in \{1,\ldots,r+1\}$ then follows
			along $\overline{P}_0$ to some $k$-simplex $\beta_{s}$, with $0\leq s < t$, and
			then leaves $P_0$  and ends at $\tau_i$ (see Figure \ref{f4}). Thus, $P$
			contains only the simplices $\beta_s$, $\beta_{s+1}$,$\ldots$,
			$\beta_{t-1}$ and
			$\alpha_{s+1}$, $\alpha_{s+2}$,$\ldots$, $\alpha_{t}$ from $P_0$.

			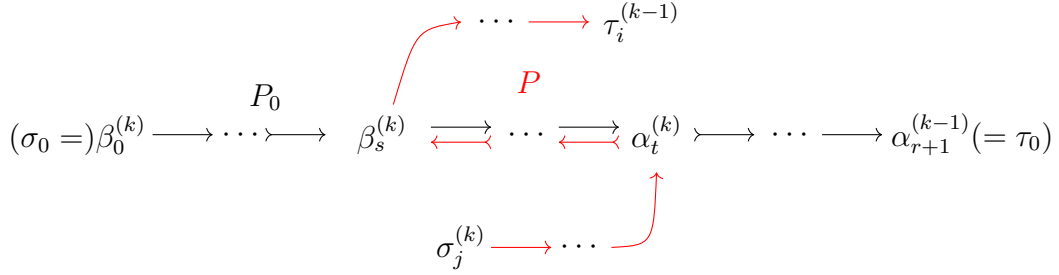
\begin{figure}[H]
				\hfill
				\begin{center}
					\begin{tikzpicture}
						\node at (0.5,0) {$(\sigma_0=)\beta_0^{(k)}$};
						\node at (2.7,0) {$\cdots$};
						\node at (4.55,0) {$\beta_s^{(k)}$};
						
						\node at (6.5,0) {$\cdots$};
						\node at (8.2,0) {$\alpha_t^{(k)}$};
						\node at (10,0) {$\cdots$};
						\node at (12.4,0) {$\alpha_{r+1}^{(k-1)}(=\tau_0)$};

						\node at (6.1,1.5) {$\cdots$}; 
						\node at (8,1.5) {$\tau_i^{(k-1)}$};

						\node at (5.6,-1.5) {$\sigma_j^{(k)}$};
						\node at (7.2,-1.5) {$\cdots$};
						
						\node at (3,.5) {$P_0$};
						\node at (6.5,.7) {\textcolor{red}{$P$}};
						
						\draw[->] (1.5,0)--(2.3,0);
						\draw[>->] (3,0)--(3.8,0);
						\draw[->] (5.2,0.1)--(6,0.1);
						\draw[->] (6.9,0.1)--(7.7,0.1);
						\draw[>->] (8.7,0)--(9.5,0);
						\draw[->] (10.4,0)--(11.2,0);
						
						\draw[->,red] (6,-1.5)--(6.8,-1.5);
						\draw[->,red] (7.6,-1.5) .. controls (8.16,-1.47).. (8.2,-0.5);

						\draw[<-<,red] (5.2,-0.1)--(6,-0.1);
						\draw[<-<, red] (6.9,-0.1)--(7.7,-0.1);

						\draw[->,red] (4.7,0.4)..controls  (5,1.4) .. (5.6,1.5);
						\draw[->,red] (6.5,1.5)--(7.3,1.5);
						
					\end{tikzpicture}
				\end{center}
				\caption{ $P$ (red) is a $\W$-trajectory from  $\sigma_j$ to $\tau_i$ that
					meets $P_0$ (black) at $\alpha_t$, follows along $\overline{P}_0$ to $\beta_{s}$
					and then leaves the trajectory $P_0$ to reach $\tau_i$.  }
				\label{f4}
			\end{figure}
			
		\end{enumerate}
	\end{observation}
	\begin{lemma} \label{l3.3}
		For any $\sigma_j \in \operatorname{Crit}^{\W}_{k}(\K)$, there is a one-to-one correspondence
		between the $\W$-trajectories from $\sigma_j$ to $\sigma_0$ and the $\V$-trajectories from $\sigma_j$ to $\tau_0$. Moreover,
		\[ \sum_{\substack{P: P \text{ is a }\\ \W\text{-traj. from } \\  \sigma_j
			\text{ to } \sigma_0 }}w(P)=-w(P_0)\cdot \sum_{\substack{P: P \text{ is a }\\
			\V\text{-traj. from } \\  \sigma_j \text{ to } \tau_0 }}w(P).\]	
		
	\end{lemma}
	\begin{proof} 	Any $\V$-trajectory $P$ from $\sigma_j$ to $\tau_0$ first
		meets $P_0$ at some $(k-1)$-simplex, say $\alpha_t$, for some $t\in \{1,\ldots,
		r+1\}$,  and then follows along $P_0$ to $\tau_0$. In this case,
		$\sigma_jP\alpha_t\overline{P}_0\sigma_0$ is a $\W$-trajectory from $\sigma_j$
		to $\sigma_0$ (see Figure \ref{f5}). This gives a one-to-one correspondence
		between the $\V$-trajectories from $\sigma_j$ to $\tau_0$ and the
		$\W$-trajectories from $\sigma_j$ to $\sigma_0$. Moreover,
		\begin{align*}
			&w(\sigma_jP\alpha_t\overline{P}_0\sigma_0)\\
			=& w(\sigma_j P \alpha_t)\cdot (-\langle \beta_{t-1}, \alpha_t \rangle
			)\cdot w(\beta_{t-1}\overline{P}_0 \sigma_0)\\
			=& w(\sigma_j P \alpha_t)\cdot (-\langle \beta_{t-1}, \alpha_t \rangle
			)\cdot w(\sigma_0P_0\beta_{t-1})\\
			=& w(\sigma_0P_0\beta_{t-1})\cdot (-\langle \beta_{t-1},
			\alpha_t\rangle)\cdot w(\sigma_jP\alpha_t) \\
			=& w(\sigma_0P_0\beta_{t-1})\cdot (-\langle \beta_{t-1},
			\alpha_t\rangle)\cdot w(\sigma_jP\alpha_t) \cdot (\langle \beta_t,\alpha_t\rangle)\cdot \langle \beta_t,\alpha_t\rangle)\cdot( 	w(\beta_tP_0\tau_0) \cdot	w(\beta_tP_0\tau_0))\\
			=&- (w(\sigma_0P_0\beta_{t-1})\cdot (-\langle \beta_{t-1}, \alpha_t\rangle \cdot \langle \beta_t,\alpha_t\rangle)
			\cdot w(\beta_tP_0\tau_0))\cdot
			(w(\sigma_jP\alpha_t)\cdot(-\langle \beta_t, \alpha_t \rangle)\cdot
			w(\beta_tP_0\tau_0))\\
			=&-w(P_0) \cdot w(P).
		\end{align*}
		\begin{figure}[H]
			\hfill
			\begin{center}
				\begin{tikzpicture}
					\node at (0,0) {$P_0:(\sigma_0=)\beta_0^{(k)}$};
					\node at (2.7,0) {$\cdots$};
					\node at (4.5,0) {$\beta_{t-1}^{(k)}$};
					\node at (6.6,0) {$\alpha_{t}^{(k-1)}$};
					\node at (8.5,0) {$\beta_{t}^{(k)}$};
					\node at (10.2,0) {$\cdots$};
					\node at (12.6,0) {$\alpha_{r+1}^{(k-1)}(=\tau_0)$};
					\node at (3.7,-1.5) {$\sigma_j^{(k)}$};
					\node at (5.4,-1.5) {$\cdots$};
					
					\node at (5,-.6)
					{\scriptsize{\textcolor{red}{$\sigma_jP\alpha_t\overline{P}_0\sigma_0$}}};
					\node at (7,-1) {\textcolor{blue}{$P$}};

					\draw[->] (1.5,0.1)--(2.3,0.1);
					\draw[>->] (3,0.1)--(3.8,0.1);
					\draw[->] (5.1,0.1)--(5.9,0.1);
					\draw[>->] (7.2,0.1)--(8,0.1);
					\draw[->] (8.9,0.1)--(9.7,0.1);
					\draw[->] (10.5,0.1)--(11.3,0.1);

					\draw[->,red] (4.2,-1.4)--(5,-1.4);
					\draw[->,red] (5.8,-1.4) .. controls (6.34,-1.46) and (6.61,-.9) ..
					(6.4,-0.5);
					
					\draw[->,blue] (4.2,-1.6)--(5,-1.6);
					\draw[->,blue] (5.8,-1.6) .. controls (6.38,-1.5) and (6.66,-1.4) ..
					(6.7,-0.5);

					\draw[<-<,red] (1.5,-0.1)--(2.3,-0.1);
					\draw[<-,red] (3,-0.1)--(3.8,-0.1);
					\draw[<-<,red] (5.1,-0.1)--(5.9,-0.1);
					
					\draw[>->,blue] (7.2,-0.1)--(8,-0.1);
					\draw[->,blue] (8.9,-0.1)--(9.7,-0.1);
					\draw[->,blue] (10.5,-0.1)--(11.3,-0.1);
				\end{tikzpicture}
			\end{center}
			\caption{$P$ (blue) is a  $\V$-trajectory from $\sigma_j$ to $\tau_0$,
				which meets $P_0$ at $\alpha_t$. Then $\sigma_jP\alpha_t\overline{P}_0\sigma_0$ (red)
				is the $\W$-trajectory corresponding to $P$.} \label{f5}
		\end{figure}
		Taking sum of the weights of all $\W$-trajectories from $\sigma_j$ to
		$\sigma_0$, we get the desired result.

	\end{proof}
	
	\begin{lemma} \label{l3.4} For any $\sigma_j \in \operatorname{Crit}^{\W}_{k}(\K)$ and $\tau_i \in \operatorname{Crit}^{\W}_{k-1}(\K)$, the sum of the weights of all
		$\W$-trajectories from $\sigma_j$ to $\tau_i$, which include some simplices from
		the trajectory $P_0$, is given by
		\[\sum_{\substack{ P:P \text{ is a } \\ \W\text{-traj. from }\\ \sigma_j
				\text{ to } \tau_i,\\ P_0\cap P\neq \emptyset}}w(P)=\sum_{\substack{ P:P \text{ is a } \\ \V\text{-traj. from}\\ \text{}
				\sigma_j \text{ to } \tau_i \text{, }\\  P_0\cap P\neq \emptyset}}w(P)-
		\left(w(P_0)\cdot \sum_{\substack{P: P \text{ is a }\\ \V\text{-traj. from } \\
				\sigma_j \text{ to } \tau_0 }}w(P) \cdot \sum_{\substack{P: P \text{ is a
				}\\ \V\text{-traj. from } \\  \sigma_0 \text{ to } \tau_i }}w(P) \right)\]
		($P_0\cap P$ denotes the set of simplices which are in both of
		the trajectories $P_0$ and $P$).
	\end{lemma}
	\begin{proof}
		Let
		\begin{enumerate}[(i)] 
    	\item$\T_1$ be the set of all $\W$-trajectories from 
			$\sigma_j$ to $\sigma_0$,
	    \item$\T_2$ be the set of all $\V$-trajectories from $\sigma_0$ to $\tau_i$,
        \item$\Z_1$ be the set of all $\W$-trajectories from $\sigma_j$ to $\tau_i$,  which include some simplices from the trajectory $P_0$, and 
	    \item$\Z_2$ be the set of all $\V$-trajectories from $\sigma_j$ to  $\tau_i$, which include some simplices from the trajectory $P_0$.
		\end{enumerate} 
		From Observation \ref{ob1}, we have $
		\Z_1 \cap \Z_2=\emptyset$. Let $P_1 \in \Z_1$ and $P_2 \in \Z_2$ be two trajectories such that $\alpha_t$ is the $(k-1)$-simplex, where $P_1$ first
		meets $P_0$, and $\beta_s$ is the $k$-simplex where
		$P_2$ leaves $P_0$ (see Figure~\ref{f6} and Figure~\ref{f7}). Now, if $s\geq t$, then 
		$\sigma_jP_1\alpha_tP_0\beta_sP_2\tau_i$ (i.e., 	$\sigma_jP_1\alpha_tP_2\tau_i$) is a trajectory in $\Z_2$ (see Figure
		\ref{f6}). Moreover,
		\begin{align}
			\nonumber&w(\sigma_jP_1\alpha_tP_0\beta_sP_2\tau_i)\\
			\nonumber=& w(\sigma_jP_1 \alpha_t)\cdot (-\langle \beta_t, \alpha_t
			\rangle) \cdot w(\beta_t P_0 \beta_s) \cdot w(\beta_s P_2 \tau_i)\\
			\nonumber =& -w(\sigma_jP_1 \alpha_t)\cdot (\langle \beta_t, \alpha_t
			\rangle) \cdot  w(\beta_t P_0 \beta_s) \cdot w(\beta_s P_2 \tau_i)\cdot (-\langle
			\beta_{t-1},\alpha_t\rangle)\cdot (-\langle
			\beta_{t-1},\alpha_t\rangle) \cdot\\
			\nonumber&w(\sigma_0P_0 \beta_{t-1}) \cdot w(\beta_{t-1}\overline{P}_0\sigma_0)~~~~(\text{since, } w(\sigma_0P_0 \beta_{t-1}) = w(\beta_{t-1}\overline{P}_0\sigma_0))\\
			\nonumber=& -(w(\sigma_jP_1 \alpha_t)\cdot (-\langle
			\beta_{t-1},\alpha_t\rangle)\cdot w(\beta_{t-1}\overline{P}_0\sigma_0)) \cdot \\
			\nonumber &(w(\sigma_0P_0 \beta_{t-1})\cdot (-\langle \beta_{t-1}, \alpha_t \rangle \cdot
			\langle \beta_t, \alpha_t \rangle)\cdot w(\beta_t P_0 \beta_s)\cdot w(\beta_s P_2 \tau_i))\\
			=&-w(P_1)\cdot w(P_2). \label{e1}
		\end{align}
		Similarly, if $s<t$, then $\sigma_jP_1\alpha_t\overline{P}_0\beta_sP_2\tau_i$
		is a trajectory in $\Z_1$ (see Figure \ref{f7}) and
		\begin{align}
			\nonumber & w(\sigma_jP_1\alpha_t\overline{P}_0\beta_sP_2\tau_i)\\
			\nonumber= & w(\sigma_jP_1 \alpha_t)\cdot (-\langle
			\beta_{t-1},\alpha_t\rangle)\cdot w(\beta_{t-1}\overline{P}_0\beta_s)\cdot 
			w(\beta_s P_2 \tau_i)\\
				\nonumber= & w(\sigma_jP_1 \alpha_t)\cdot (-\langle
			\beta_{t-1},\alpha_t\rangle)\cdot w(\beta_{t-1}\overline{P}_0\beta_s)\cdot 
			w(\beta_s P_2 \tau_i)\cdot(w(\beta_{s}\overline{P}_0\sigma_0)\cdot w(\sigma_0P_0 \beta_{s}))\\
			\nonumber=& (w(\sigma_jP_1 \alpha_t)\cdot (-\langle
			\beta_{t-1},\alpha_t\rangle)\cdot w(\beta_{t-1}\overline{P}_0\beta_s)\cdot
			w(\beta_{s}\overline{P}_0\sigma_0))\cdot (w(\sigma_0P_0 \beta_{s})\cdot  w(\beta_sP_2 \tau_i))\\
			= &w(P_1)\cdot w(P_2).    \label{e2}
		\end{align}
		We now define a map $\phi:\T_1 \times \T_2 \rightarrow \Z_1 \sqcup
		\Z_2$ as follows. For $(P_1,P_2) \in \T_1 \times \T_2$, such that $\alpha_t$ is the simplex where
		$P_1$ meets the trajectory $P_0$, and $\beta_s$ 
		is the simplex where $P_2$ leaves the trajectory $P_0$,
		\[\phi((P_1,P_2))=\begin{cases}
			\sigma_jP_1\alpha_tP_0\beta_sP_2\tau_i, &\text{ if } s\geq t,\\
			\sigma_jP_1\alpha_t\overline{P}_0\beta_sP_2\tau_i, &\text{ if } s<t.
		\end{cases} \]
		\begin{figure}[H]
			\hfill
			\begin{center}
				\begin{tikzpicture}
					\node at (-0.5,0) {$P_0:(\sigma_0=)\beta_0^{(k)}$};
					\node at (2.3,0) {$\cdots$};
					\node at (4.4,0) {$\alpha_t^{(k-1)}$};
					
					\node at (6.5,0) {$\cdots$};
					\node at (8.3,0) {$\beta_s^{(k)}$};
					\node at (10.1,0) {$\cdots$};
					\node at (12.6,0) {$\alpha_{r+1}^{(k-1)}(=\tau_0)$};

					\node at (10.2,1.5) {$\cdots$}; 
					\node at (12.2,1.5) {$\tau_i^{(k-1)}$};

					\node at (1.8,-1.5) {$\sigma_j^{(k)}$};
					\node at (3.4,-1.5) {$\cdots$};

					\node at (3.8,-1) {\textcolor{red}{$P_1$}};
					\node at (8.7,1) {\textcolor{blue}{$P_2$}};
					\node at (7,-.8) {\scriptsize{$\sigma_jP_1\alpha_tP_0\beta_sP_2\tau_i$}};
					
					\draw[->] (1,0)--(1.8,0);
					\draw[->] (2.8,0)--(3.6,0);
					\draw[>->] (5.2,0)--(6,0);
					
					\draw[>->] (6.9,0)--(7.7,0);
					\draw[->] (8.8,0)--(9.6,0);
					\draw[->] (10.5,0)--(11.3,0);
					
					\draw[->,red] (2.2,-1.5)--(3,-1.5);
					\draw[->,red] (3.8,-1.5) .. controls (4.36,-1.47).. (4.4,-0.5);
					
					\draw[<-<,red] (1,-0.15)--(1.8,-0.15);
					\draw[<-<,red] (2.8,-0.15)--(3.6,-0.15);

					\draw[->,blue] (8.8,0.1)..controls (9.1,0.3) and (9.1,1.4) .. (9.7,1.5);
					\draw[->,blue] (10.7,1.5)--(11.5,1.5);
					
					\draw[->,blue] (1,0.15)--(1.8,0.15);
					\draw[->,blue] (2.8,0.15)--(3.6,0.15);
					\draw[>->,blue] (5.2,0.15)--(6,0.15);
					\draw[>->,blue] (6.9,0.15)--(7.7,0.15);
					
					\draw[->,rounded corners=13pt,thick,dotted]
					(2,-1.9)--(4.7,-1.9)--(4.8,-0.5)--(9.15,-.5)--(9.7,1.1)--(11.6,1.1);
					
				\end{tikzpicture}
			\end{center}
			\caption{$P_1$ (red) is a $\W$-trajectory from $\sigma_j$ to $\sigma_0$
				which meets $P_0$ at $\alpha_t$. $P_2$ (blue) is a $\V$-trajectory from
				$\sigma_0$ to $\tau_i$ which leaves $P_0$ at $\beta_s$. For $s\geq t$,
				$\phi((P_1,P_2))=\sigma_jP_1\alpha_tP_0\beta_sP_2\tau_i$ (dotted black).}
			\label{f6}
		\end{figure}
		\begin{figure}[H]
			\hfill
			\begin{center}
				\begin{tikzpicture}
					\node at (-0.5,0) {$P_0:(\sigma_0=)\beta_0^{(k)}$};
					\node at (2.3,0) {$\cdots$};
					\node at (4.2,0) {$\beta_s^{(k)}$};
					
					\node at (6.2,0) {$\cdots$};
					\node at (8.2,0) {$\alpha_t^{(k-1)}$};
					\node at (10.1,0) {$\cdots$};
					\node at (12.6,0) {$\alpha_{r+1}^{(k-1)}(=\tau_0)$};

					\node at (6.2,1.5) {$\cdots$}; 
					\node at (8.2,1.5) {$\tau_i^{(k-1)}$};

					\node at (5.6,-1.5) {$\sigma_j^{(k)}$};
					\node at (7.2,-1.5) {$\cdots$};

					\node at (7.7,-1) {\textcolor{red}{$P_1$}};
					\node at (4.7,1) {\textcolor{blue}{$P_2$}};
					\node at (9.8,-.8) {\scriptsize{{$
								\sigma_jP_1\alpha_t\overline{P}_0\beta_sP_2\tau_i$}}};
					
					\draw[->] (1,0)--(1.8,0);
					\draw[>->] (2.8,0)--(3.6,0);
					\draw[->] (4.8,0)--(5.6,0);
					\draw[->] (6.6,0)--(7.4,0);
					\draw[>->] (8.8,0)--(9.6,0);
					\draw[->] (10.5,0)--(11.3,0);
					
					\draw[->,red] (6,-1.5)--(6.8,-1.5);
					\draw[->,red] (7.6,-1.5) .. controls (8.16,-1.47).. (8.2,-0.5);
					
					\draw[<-<,red] (1,-0.15)--(1.8,-0.15);
					\draw[<-,red] (2.8,-0.15)--(3.6,-0.15);
					\draw[<-<,red] (4.8,-0.15)--(5.6,-0.15);
					\draw[<-<,red] (6.6,-0.15)--(7.4,-0.15);
					
					\draw[->,blue] (4.8,0.1)..controls (5.1,0.3) and (5.1,1.4) .. (5.7,1.5);
					\draw[->,blue] (6.7,1.5)--(7.5,1.5);
					
					\draw[->,blue] (1,0.15)--(1.8,0.15);
					\draw[->,blue] (2.8,0.15)--(3.6,0.15);

					\draw[->,rounded corners=12pt,thick,dotted]
					(5.6,-1.9)--(8.5,-1.9)--(8.8,0.4)--(5.3,.4)--(5.5,1.15)--(8.2,1.15);
					
				\end{tikzpicture}
			\end{center}
			\caption{$P_1$ (red) is a $\V$-trajectory from $\sigma_j$ to $\sigma_0$
				which meets $P_0$ at $\alpha_t$. $P_2$ (blue) is a $\W$-trajectory from
				$\sigma_0$ to $\tau_i$ which leaves $P_0$ at $\beta_s$. For $s < t$,
				$\phi((P_1,P_2))= \sigma_jP_1\alpha_t\overline{P}_0\beta_sP_2\tau_i$ (dotted
				black).} \label{f7}
		\end{figure}
		We note that $\phi$ is a bijection, and thus it follows that
		\begin{align*}
			\sum_{\substack{(P_1,P_2)\in \T_1\times \T_2,\\ \phi((P_1,P_2))\in \Z_1}}
			w(\phi((P_1,P_2)))=& \sum_{P \in \Z_1}w(P)\\
			\text{  and }\sum_{\substack{(P_1,P_2)\in \T_1\times \T_2,\\
					\phi((P_1,P_2))\in \Z_2}} w(\phi((P_1,P_2)))=& \sum_{P \in \Z_2}w(P).
		\end{align*}
		So, we get
		\begin{align}
		\nonumber	\sum_{P \in \Z_1}w(P)-\sum_{P \in \Z_2}w(P)
			= &\sum_{\substack{(P_1,P_2)\in \T_1\times \T_2,\\ \phi((P_1,P_2))\in
					\Z_1}} w(\phi((P_1,P_2)))- \sum_{\substack{(P_1,P_2)\in \T_1\times \T_2,\\
					\phi((P_1,P_2))\in \Z_2}} w(\phi((P_1,P_2)))\\
		\nonumber	= &\sum_{\substack{(P_1,P_2)\in \T_1\times \T_2,\\ \phi((P_1,P_2))\in
					\Z_1}} w(\sigma_jP_1\alpha_t\overline{P}_0\beta_sP_2\tau_i)-
			\sum_{\substack{(P_1,P_2)\in \T_1\times \T_2,\\ \phi((P_1,P_2))\in \Z_2}}
			w(\sigma_jP_1\alpha_tP_0\beta_sP_2\tau_i)\\
		\nonumber	= &\sum_{\substack{(P_1,P_2)\in \T_1\times \T_2,\\ \phi((P_1,P_2))\in \Z_1}}
			w(P_1)\cdot w(P_2) + \sum_{\substack{(P_1,P_2)\in \T_1\times \T_2, \\
					\phi((P_1,P_2))\in \Z_2}} w(P_1)\cdot w(P_2) \left(\substack{\text{from Eqn.}~
				\eqref{e1}\\ \text{ and Eqn.}~ \eqref{e2}}\right)\\
		\nonumber	= &\sum_{\substack{(P_1,P_2)\in \T_1\times \T_2}} w(P_1)\cdot w(P_2)\\
			=  &\sum_{P_1\in \T_1} w(P_1)\cdot \sum_{P_2\in \T_2} w(P_2). \label{eq3}
		\end{align}
		Therefore, we get
		\begin{align*}
			\sum_{\substack{ P:P \text{ is a } \\ \W\text{-traj. from }\\ \sigma_j
					\text{ to } \tau_i,\\ P_0\cap P\neq \emptyset}}w(P) &=\sum_{P \in \Z_1}w(P)\\
			 &= \sum_{P \in \Z_2}w(P)+\sum_{P_1 \in \T_1} w(P_1)\cdot \sum_{P_2 \in
				\T_2} w(P_2) \text{ (from Equation~(\ref{eq3}))}\\
		     &=  \sum_{P \in \Z_2}w(P)+\left(-w(P_0)\cdot \sum_{\substack{P: P \text{ is a } \\ \V\text{-traj.
						from } \\  \sigma_j \text{ to } \tau_0 }}w(P)\right)\cdot \sum_{P_2 \in \T_2}
		    	w(P_2) \text{    (from Lemma~\ref{l3.3})}\\
		     &=\sum_{\substack{ P:P \text{ is a } \\ \V\text{-traj. from }\\ \sigma_j
					\text{ to } \tau_i,\\ P_0\cap P\neq \emptyset}}w(P)- \left(w(P_0)\cdot
			\sum_{\substack{P: P \text{ is a }\\ \V\text{-traj. from } \\   \sigma_j \text{
						to } \tau_0 }}w(P) \cdot \sum_{\substack{P: P \text{ is a }\\ \V\text{-traj.
						from } \\ \sigma_0 \text{ to } \tau_i }}w(P) \right).
		\end{align*}
		
	\end{proof}
	
	We now proceed to	the proof of Proposition \ref{P3.1.}.
	\begin{proof}[Proof of Proposition~\ref{P3.1.}] 
		We have, for all $j\in \{0, \dots , n\}$, $\partial_k^{\V}(\sigma_j)=
		\sum_{i=0}^m a_{ij}\tau_i$ , where
		\[a_{ij}= \sum_{\substack{ P:P \text{ is a }\\ \V\text{-traj. from} \\ 
				\sigma_j \text{ to } \tau_i }}w(P).\]
		Since, $P_0$ is the unique $\V$-trajectory from $\sigma_0$ to $\tau_0$,
		$a_{00}=w(P_0)$.
		For any $\W$-critical $k$-simplex $\sigma_j$ and
		any $\W$-critical $(k-1)$-simplex $\tau_i$, the
		coefficient of $\tau_i$ in $\partial_k^{\W}(\sigma_j)$ is
		\[\sum_{\substack{ P:P \text{ is a }\\ \W\text{-traj. from } \\  \sigma_j
				\text{ to } \tau_i }}w(P).\]

		Now,
		\begin{align*}
			\sum_{\substack{ P:P \text{ is a }\\ \W\text{-traj. from} \\  \sigma_j
					\text{ to } \tau_i }}w(P) 
			=&\sum_{\substack{ P:P \text{ is a } \\ \W\text{-traj. from }\\  \sigma_j
					\text{ to } \tau_i, \\ P_0\cap P = \emptyset}}w(P) + \sum_{\substack{ P:P \text{
						is a } \\ \W\text{-traj. from }\\  \sigma_j \text{ to } \tau_i, \\ P_0\cap P
					\neq \emptyset}}w(P)\\
			= &\sum_{\substack{ P:P \text{ is a } \\ \V\text{-traj. from }\\  \sigma_j
					\text{ to } \tau_i, \\ P_0\cap P = \emptyset}}w(P) + \sum_{\substack{ P:P \text{
						is a } \\ \W\text{-traj. from }\\  \sigma_j \text{ to } \tau_i, \\ P_0\cap P
					\neq \emptyset}}w(P) \text{  (by Observation \ref{ob1}.(1))}\\
			= &\sum_{\substack{ P:P \text{ is a } \\ \V\text{-traj. from }\\  \sigma_j
					\text{ to } \tau_i, \\ P_0\cap P = \emptyset}}w(P) + \sum_{\substack{ P:P \text{
						is a } \\ \V\text{-traj. from }\\  \sigma_j \text{ to } \tau_i, \\ P_0\cap P
					\neq \emptyset}}w(P)\\
			&- \left(w(P_0)\cdot \sum_{\substack{ P:P \text{ is a }\\ \V\text{-traj.
						from } \\ \sigma_j \text{ to } \tau_0 }}w(P) \cdot \sum_{\substack{ P:P
					\text{ is a }\\ \V\text{-traj. from } \\  \sigma_0 \text{ to } \tau_i}}
			w(P)\right) \text{  (by Lemma \ref{l3.4})}\\ 
			= &\sum_{\substack{ P:P \text{ is a } \\ \V\text{-traj. from }\\  \sigma_j
					\text{ to } \tau_i}}w(P) - \left(w(P_0)\cdot \sum_{\substack{ P:P \text{ is a
					}\\ \V\text{-traj. from } \\ \sigma_j \text{ to } \tau_0 }}w(P) \cdot
			\sum_{\substack{ P:P \text{ is a }\\ \V\text{-traj. from } \\  \sigma_0 \text{
						to } \tau_i}} w(P)\right)\\
			=& a_{ij}-a_{00}a_{0j}a_{i0}.
		\end{align*}
		
	\end{proof}	
	\begin{remark} \label{r3.5}
		Considering the matrix representation of $\partial^{\V}_k$, we get the following.
		\begin{align*}
		& \partial^{\V}_k  =
		\begin{pNiceMatrix}[first-col,first-row, last-col=5]
			   ~   &\sigma_0&\sigma_1& \ldots &\sigma_n&  \\
			\tau_0 & a_{00} & a_{01} & \ldots & a_{0n} &  \textcolor{blue}{R_0}\\
			\tau_1 & a_{10} & a_{11} & \ldots & a_{1n} &  \textcolor{blue}{R_1}\\
			\vdots	~   & \vdots & \vdots & ~      & \vdots & ~\textcolor{blue}{\vdots}\\
			\tau_i & a_{i0} \Block[draw,rounded-corners,color=blue]{1-4}{} & a_{i1} & \ldots & a_{in} & \textcolor{blue}{R_i}\\ 
			\vdots ~& \vdots & \vdots & ~ & \vdots & ~\textcolor{blue}{\vdots}\\
			\tau_m & a_{m0} & a_{m1} & \ldots & a_{mn} & \textcolor{blue}{R_m}
		\end{pNiceMatrix}
          \xrightarrow{a_{00}\cdot R_0}
	    \begin{pNiceMatrix}[first-col,first-row, last-col=5]
	    	~ &\sigma_0 & \sigma_1 & \ldots & \sigma_n &  \\
	    	\tau_0 & 1 & a_{00}a_{01} & \ldots & a_{00}a_{0n} & \textcolor{blue}{R_0}\\
	    	\tau_1 & a_{10} & a_{11} & \ldots & a_{1n} & \textcolor{blue}{R_1} \\
	    	\vdots~ & \vdots & \vdots & ~ & \vdots & ~\textcolor{blue}{\vdots}\\
	    	\tau_i & a_{i0} \Block[draw,rounded-corners,color=blue]{1-4}{} & a_{i1} & \ldots & a_{in} & \textcolor{blue}{R_i}\\ 
	    	\vdots~ & \vdots & \vdots & ~ & \vdots & ~\textcolor{blue}{\vdots}\\
	    	\tau_m & a_{m0} & a_{m1} & \ldots & a_{mn} & \textcolor{blue}{R_m}
	    \end{pNiceMatrix}
        ~~~\left(\substack{\text{since, }\\ a_{00}=\pm 1}\right)\\
        \end{align*}
        \begin{align*}
           & \xrightarrow [\text{for each } i\geq 1]{ R_i- a_{i0}\cdot R_0} 
    	 \begin{pNiceMatrix}[first-col,first-row, last-col=7]
    		~ &\sigma_0 & \sigma_1 & \ldots & \sigma_j & \ldots & \sigma_n &  \\
    		\tau_0 & 1 & a_{00}a_{01} & \ldots &  a_{00}a_{0j}& \ldots & a_{00}a_{0n} & \textcolor{blue}{R_0}\\ 
    		\tau_1 & 0 & a_{11}-a_{10} \cdot a_{00}a_{01} \Block[draw,rounded-corners,color=black]{5-5}{} & \ldots & a_{1j}-a_{10} \cdot a_{00}a_{0j} & \ldots & a_{1n}-a_{10} \cdot a_{00}a_{0n} & \textcolor{blue}{R_1}\\
    	    \vdots ~	& \vdots & \vdots & ~ & \vdots &  ~ &\vdots & ~\textcolor{blue}{\vdots}\\
    		\tau_i & 0 \Block[draw,rounded-corners,color=blue]{1-6}{} & a_{i1}-a_{i0} \cdot a_{00}a_{01} & \ldots & a_{ij}-a_{i0} \cdot a_{00}a_{0j}& \ldots & a_{in}-a_{i0} \cdot a_{00}a_{0n} & \textcolor{blue}{R_i}\\ 
    		\vdots ~ & \vdots & \vdots & ~ & \vdots & ~ &\vdots & ~\textcolor{blue}{\vdots}\\
    		\tau_m & 0 & a_{m1}-a_{m0} \cdot a_{00}a_{01} & \ldots& a_{mj}-a_{m0} \cdot a_{00}a_{0j} & \cdots & a_{mn}-a_{m0} \cdot a_{00}a_{0n} & \textcolor{blue}{R_m}
    		\CodeAfter \tikz [red] \draw (1.5-|1) -- (1.5-|last) (1-|1.5) -- (last-|1.5) ;
    	\end{pNiceMatrix}\\
       & \begin{tikzpicture}
        	\node at (0,0) {~};
        	\node at (9.5,0) {\rotatebox[origin=c]{270}{=}};
        	\node at (9.5,-.5) {$\partial^{\W}_k$};
        \end{tikzpicture}
    	\end{align*}	
        
	\end{remark}
	\subsection{Analogous result for the co-Morse complex}
	
	For a simplicial complex $\K$ and a gradient vector field $\V$ defined on it, a
	co-$\V$-trajectory from a $(q-1)$-simplex $\tau$ to a $q$-simplex $\sigma$ is
	just a $\V$-trajectory from $\sigma$ to $\tau$ with the order of the sequence
	reversed. Moreover, the weights of the co-$\V$-trajectories and the corresponding
	$\V$-trajectories are same. So we can say,
	\[ \sum_{\substack {P:P \text{ is a }\\ \V \text{-traj. from } \\ \sigma \text{
				to } \tau }} w(P) =  \sum_{\substack {Q:Q \text{ is a }\\
			\text{co-}\V\text{-traj. from } \\ \tau \text{ to } \sigma }} w(\overline{P}).\]
	
	This implies, for any $\sigma\in \operatorname{Crit}_q^{\V}(\K)$ and $\tau\in
	\operatorname{Crit_{q-1}^{\V}}(\K)$, the coefficient of $\tau$ in
	$\partial_q^{\V}(\sigma)$ and the coefficient of $\sigma$ in
	$\delta_{q}^{\V}(\tau)$ are equal. Hence, the matrix representation of
	$\delta_{q}^{\V}$ is just the transpose of the matrix of
	$\partial_q^{\V}$. Therefore, by an analogous argument used while proving
	Theorem~ \ref{T1.1}, we obtain the following for the coboundary operators of
	the co-Morse complex of $\K$.
	\begin{theorem} \label{T3.3}
		Let $\K$ be a $d$-dimensional simplicial complex with an assigned gradient
		vector field $\V$. Let $(\sigma_0^{(k)},\tau_0^{(k-1)})$ be a cancellable
		critical pair,  for some $k \in \{1, \dots, d\}$, and $\W$ be the gradient
		vector field obtained by cancelling $(\sigma_0^{(k)},\tau_0^{(k-1)})$ from $\V$.
		Let $\delta_q^{\V}: C_{q-1}^{\V}(\K) \rightarrow C_{q}^{\V}(\K)$ and
		$\delta_q^{\W}: C_{q-1}^{\W}(\K) \rightarrow C_{q}^{\W}(\K)$ be the $q$-th
		coboundary maps of the co-Morse complexes of $\K$ corresponding to $\V$ and
		$\W$, respectively. Then the following hold.
		\begin{enumerate}[(1)]
			\item For $q>k+1$ or $q<k-1$,
			\[ \delta_q^{\W}= \delta_q^{\V}.\]
			
			\item  If  $\operatorname{Crit}^{\W}_{k-1}(\K)= \{ \tau_1, \dots, \tau_m \}$
			and for any $\beta \in \operatorname{Crit}_{k-2}^{\W}$,
			$\delta_{k-1}^{\V}(\beta)= \sum_{i=0}^{m} b_i\tau_i$, then, 
			\[\delta_{k-1}^{\W}(\beta)= \sum_{i=1}^{m} b_i\tau_i.\]
			
			\item The boundary operator $\delta_{k+1}^{\W}$ is the restriction of
			$\delta_{k+1}^{\V}$ to the subgroup $C_{k}^{\W}(\K)$, i.e.,
			\[\delta_{k+1}^{\W}=\delta_{k+1}^{\V}\big|_{C_{k}^{\W}(\K)}.\]

			\item Let $\operatorname{Crit}^{\V}_{k}(\K)= \{ \sigma_0, \sigma_1, \dots,
			\sigma_n \}$, 
			$\operatorname{Crit}^{\V}_{k-1}(\K)= \{ \tau_0, \tau_1, \dots, \tau_m \}$, and
			for all $i\in \{0, \dots , m\}$, $\delta_k^{\V}(\tau_i)= \sum_{j=0}^n
			a_{ji}\sigma_j$ . Then, for all $j\in\{1, \dots , m\}$,
			\[\delta_k^{\W}(\tau_i)= \sum_{j=1}^n (a_{ji}-a_{00}a_{0i}a_{j0}) \cdot
			\sigma_j.\]
		\end{enumerate}

	\end{theorem}
\section*{Acknowledgements}
 The authors would like to thank the first anonymous reviewer of the journal version of this article for corrections and other valuable suggestions which substantially improved the presentation of the article.

	\begin{appendices} 
	\section{An example of simultaneous cancellations} \label{A1}
	In \cite[Section 4]{mondal}, the (discrete) Morse homology groups of the \emph{matching complex} of the complete graph of order 7 are computed with respect to a `near optimal' gradient vector field $\V$ (denoted by $\M^*$ in \cite{mondal}). There are four $\V$-critical $1$-simplices and twenty four $\V$-critical $2$-simplices. If $\eta_1,\eta_2,\ldots,\eta_{24}$ are the critical $2$-simplices and $\sigma_1,\ldots,\sigma_4$ are the critical $1$-simplices, then  $\partial_1^{\V}(\sigma_i)=0$, for all $i \in \{1,\ldots,4\}$, and the following table (reproduced from \cite[Table~1]{mondal}) represents the images of the $2$-simplices under the boundary map $\partial_2^{\V}$. 
	
	\begin{table}[htbp]
		\centering
			\begin{tblr}{|[1pt]c|c|[1pt]c|c|[1pt]c|c|[1pt]c|c|[1pt]}
			\hline[1pt]
			$\eta$ & $\partial_2^{\V}(\eta)$ & $\eta$ & $\partial_2^{\V}(\eta)$ & $\eta$ & $\partial_2^{\V}(\eta)$ & $\eta$ & $\partial_2^{\V}(\eta)$\\
			\hline[1pt]
			
			$\eta_{1}$ & $\sigma_2-\sigma_3$&
			
			$\eta_{2}$ & $\sigma_1-\sigma_2-\sigma_3$ &
			
			$\eta_{3}$ & $\sigma_1-\sigma_3+\sigma_4$ &
			
			$\eta_{4}$ & $\sigma_2+\sigma_3-\sigma_4$\\
			\hline
			
			$\eta_{5}$ & $\sigma_1-\sigma_4$&
			
			$\eta_{6}$ & $-\sigma_1+\sigma_2-\sigma_4$ &
			
			$\eta_{7}$ & $-\sigma_1+\sigma_4$ &
			
			$\eta_{8}$ & $-\sigma_2+\sigma_3$\\
			\hline
			
			$\eta_{9}$ & $-\sigma_1+\sigma_2+\sigma_3$ &
			
			$\eta_{10}$ & $-\sigma_1+\sigma_3-\sigma_4$ &
			
			$\eta_{11}$ & $-\sigma_1+\sigma_2+\sigma_3$ &
			
			$\eta_{12}$ & $\sigma_1-\sigma_2+\sigma_4$\\
			\hline
			
			$\eta_{13}$ & $-\sigma_2-\sigma_3+\sigma_4$ &
			
			$\eta_{14}$ & $-\sigma_2-\sigma_3+\sigma_4$ &
			
			$\eta_{15}$ & $\sigma_2-\sigma_3$ &
			
			$\eta_{16}$ & $\sigma_1-\sigma_4$\\
			\hline
			
			$\eta_{17}$ & $\sigma_1-\sigma_2+\sigma_4$ &
			
			$\eta_{18}$ & $-\sigma_1+\sigma_4$ &
			
			$\eta_{19}$ & $-\sigma_1+\sigma_2-\sigma_4$ &
			
			$\eta_{20}$ & $\sigma_1-\sigma_2-\sigma_3$\\
			\hline
			
			$\eta_{21}$ & $-\sigma_1+\sigma_3-\sigma_4$ &
			
			$\eta_{22}$ & $-\sigma_2+\sigma_3$ &
			
			$\eta_{23}$ & $\sigma_2+\sigma_3-\sigma_4$ &
			
			$\eta_{24}$ & $\sigma_1-\sigma_3+\sigma_4$\\
			\hline[1pt]
		\end{tblr}
		\caption{Images of all critical $2$-simplices under the boundary operator $\partial_2^{\V}$.}\label{boundary-table} \label{Table1}
	\end{table}
    From these we deduce that the first Morse homology group is $\mathbb{Z}_3$.
     
    In \cite[Subsection~4.3]{mondal}, it is shown that the critical pairs  $(\eta_{8}^{(2)},\sigma_{3}^{(1)})$ and $(\eta_{18}^{(2)},\sigma_4^{(1)})$ satisfy Theorem~\ref{t2.9}, and thus, they are simultaneously cancellable. Let $\W_1$ be the gradient vector field obtained after cancelling $(\eta_{8}^{(2)},\sigma_{3}^{(1)})$ from $\V$, and $\W_2$ be the gradient vector field obtained after cancelling $(\eta_{18}^{(2)},\sigma_4^{(1)})$ from $\W_1$.

    We obtain the boundary operators $\partial_2^{(\W_1)}$ and $\partial_2^{(\W_2)}$, from $\partial_2^{(\V)}$ and $\partial_2^{(\W_2)}$, respectively, by a sequence of elementary row operations (see Remark~\ref{r3.5}) as follows.

\[(\partial_2^{\V})^T = 
\begin{pNiceMatrix}[first-row,first-col,r]
	      & \sigma_1 & \sigma_2 & \sigma_3 & \sigma_4 \\
	\eta_1& 0&  1& -1 \Block[draw,rounded-corners,color=red]{*-1}{} &  0\\
	\eta_2& 1& -1& -1&  0\\
	\eta_3& 1&  0& -1&  1\\
	\eta_4& 0&  1&  1& -1\\
	\eta_5&  1&  0&  0& -1\\
	\eta_6& -1&  1&  0& -1\\
	\eta_7& -1&  0&  0&  1\\
	\eta_8& 0 \Block[draw,rounded-corners,color=red]{1-*}{} & -1&  1&  0\\
	\eta_9& -1&  1&  1&  0\\
	\eta_{10}& -1&  0&  1& -1\\
	\eta_{11}&   -1&  1&  1&  0\\
	\eta_{12}&  1& -1&  0&  1\\
	\eta_{13}&   0& -1& -1&  1\\
	\eta_{14}&   0& -1& -1&  1\\
	\eta_{15}&    0&  1& -1&  0\\
	\eta_{16}&    1&  0&  0& -1\\
	\eta_{17}&   1& -1&  0&  1\\
	\eta_{18}&   -1&  0&  0&  1\\
	\eta_{19}&   -1&  1&  0& -1\\
	\eta_{20}&   1& -1& -1&  0\\
	\eta_{21}&  -1&  0&  1& -1\\
	\eta_{22}&  0& -1&  1&  0 \\
	\eta_{23}&  0&  1&  1& -1 \\
	\eta_{24}&  1&  0& -1&  1	
\end{pNiceMatrix},
(\partial_2^{\W_1})^T = 
\begin{pNiceMatrix}[first-row,first-col,r]
	~ &\sigma_1 & \sigma_2 & \sigma_4 \\
	\eta_1&	0&  0&  0 \Block[draw,rounded-corners,color=red]{*-1}{} \\
	\eta_2&	1& -2&  0\\
	\eta_3&	1&  -1&  1\\
	\eta_4&	0&  2& -1\\
	\eta_5&	1&  0& -1\\
	\eta_6&	-1&  1& -1\\
	\eta_7&	-1&  0&  1\\
	\eta_9&	-1&  2&  0\\
	\eta_{10}&	-1&  1& -1\\
	\eta_{11}&	-1&  2&  0\\
	\eta_{12}&	1& -1&  1\\
	\eta_{13}&	0& -2&  1\\
	\eta_{14}&	0& -2&  1\\
	\eta_{15}&	0&  0&  0\\
	\eta_{16}&	1&  0& -1\\
	\eta_{17}&	1& -1&  1\\
	\eta_{18}&	-1 \Block[draw,rounded-corners,color=red]{1-*}{}&  0&  1\\
	\eta_{19}&	-1&  1& -1\\
	\eta_{20}&	1& -2&  0\\
	\eta_{21}&	-1&  -1& -1\\
	\eta_{22}&	0& 0& 0\\
	\eta_{23}&	0&  2& -1\\
	\eta_{24}&	1&  -1& 1
\end{pNiceMatrix},
(\partial_2^{\W_2})^T =
\begin{pNiceMatrix}[first-row,first-col,r]
 ~ &\sigma_1 & \sigma_2 \\
	\eta_1&	0&  0\\
	\eta_2&	1& -2\\
	\eta_3&	2&  -1\\
	\eta_4&	-1&  2\\
	\eta_5&	0&  0\\
	\eta_6&	-2&  1\\
	\eta_7&	0&  0\\
	\eta_9&	-1&  2\\
	\eta_{10}&	-2&  1\\
	\eta_{11}&	-1&  2\\
	\eta_{12}&	2& -1\\
	\eta_{13}&	1& -2\\
	\eta_{14}&	1& -2\\
	\eta_{15}&	0&  0\\
	\eta_{16}&	0&  0\\
	\eta_{17}&	2& -1\\
	\eta_{19}&	-2&  1\\
	\eta_{20}&	1& -2\\
	\eta_{21}&	-2&  -1\\
	\eta_{22}&	0& 0\\
	\eta_{23}&	-1&  2\\
	\eta_{24}&	2&  -1
\end{pNiceMatrix}.
\]	
 
 We may check that the first Morse homology group, with respect to both $\W_1$ and $\W_2$, turns out to be $\mathbb{Z}_3$ as expected.
\end{appendices}

\end{document}